\documentclass[10pt]{article}
\usepackage{color}
\usepackage{amssymb}
\usepackage{amsmath}
\usepackage{latexsym}
\usepackage{mathrsfs}
\usepackage{amsthm}
\usepackage{enumerate}

\allowdisplaybreaks[2]

\numberwithin{equation}{section}
\newtheorem{thm}[equation]{Theorem}
\newtheorem{prop}[equation]{Proposition}

\newtheorem{rem}[equation]{Remark}
\newtheorem{lem}[equation]{Lemma}
\newtheorem{corol}[equation]{Corollary}

\providecommand{\norm}[1]{\left\lVert#1\right\rVert}

\title{A few shape optimization results for a\\ biharmonic Steklov problem}

\author{Davide Buoso\footnote{Corresponding author.}\   and Luigi Provenzano}

\date{\ }

\begin{document}

\newcommand{\rea}{\mathbb{R}}
\newcommand{\eps}{\varepsilon}

\maketitle

\noindent
{\bf Abstract:}
We derive the equation of a free vibrating thin plate whose mass is concentrated at the boundary, namely a Steklov problem for the biharmonic operator. We provide Hadamard-type formulas for the shape derivatives of the corresponding eigenvalues and prove that balls are critical domains under volume constraint. Finally, we prove an isoperimetric inequality for the first positive eigenvalue.

\vspace{11pt}

\noindent
{\bf Keywords:} biharmonic operator, Steklov boundary conditions, eigenvalues, isovolumetric perturbations, isoperimetric inequality.

\vspace{6pt}
\noindent
{\bf 2010 Mathematics Subject Classification:} Primary 35J30; Secondary 35C05, 35P15, 49R05, 74K20.


\section{Introduction}

Let $\Omega$ be a bounded domain (i.e., a bounded connected open set) in $\mathbb{R}^N$ of class $C^1$, $N\ge 2$. We consider the following Steklov problem for the biharmonic operator

\begin{equation}\label{Steklov-Bi}
\left\{\begin{array}{ll}
\Delta^2 u -\tau\Delta u=0 ,\ \ & {\rm in}\  \Omega,\\
\frac{\partial^2 u}{\partial \nu^2 }=0 ,\ \ & {\rm on}\  \partial \Omega,\\
\tau\frac{\partial u}{\partial \nu }-{\rm div}_{\partial\Omega}\big(D^2u.\nu\big)-\frac{\partial\Delta u}{\partial\nu}=\lambda u ,\ \ & {\rm on}\  \partial \Omega,
\end{array}\right.
\end{equation}
in the unknowns $u$ (the eigenfunction), $\lambda$ (the eigenvalue), where $\tau>0$ is a fixed positive constant, $\nu$ denotes the outer unit normal to $\partial\Omega$, ${\rm div}_{\partial\Omega}$ denotes the tangential divergence operator and $D^2u$ the Hessian matrix of $u$. For $N=2$, this problem is related to the study of the vibrations of a thin elastic plate with a free frame and mass concentrated at the boundary. The spectrum consists of a diverging sequence of eigenvalues of finite multiplicity
$$
0=\lambda_1(\Omega)<\lambda_2(\Omega)\le \dots \le\lambda_j(\Omega)\le\cdots,
$$
where we agree to repeat the eigenvalues according to their multiplicity. We note that problem (\ref{Steklov-Bi}) is the analogue for the biharmonic operator of the classical Steklov problem for the Laplace operator, namely
\begin{equation}\label{Steklov-2}
\left\{\begin{array}{ll}
\Delta u =0 ,\ \ & {\rm in}\  \Omega,\\
\frac{\partial u}{\partial \nu }=\lambda u ,\ \ & {\rm on}\  \partial \Omega,
\end{array}\right.
\end{equation}
which models the vibrations of a free membrane with mass concentrated at the boundary. Problem (\ref{Steklov-2}) was first considered by Steklov in \cite{stek}, where the author provided a physical derivation (see also \cite{proz}). We refer to \cite{lambertisteklov} for related problems, and to \cite{gir} for a recent survey on the subject.

In this paper we are interested in the dependence of the eigenvalues $\lambda_j(\Omega)$ of problem (\ref{Steklov-Bi}) on the domain $\Omega$. Domain perturbation problems have been widely studied in the case of the Laplace operator subject to different homogeneous boundary conditions (Dirichlet, Neumann, Steklov, etc.), in particular for shape optimization problems. We recall for instance the celebrated Faber-Krahn inequality, which says that the ball minimizes the first eigenvalue of the Dirichlet Laplacian among all  domains with fixed measure (see \cite{faber,krahn}). Similar results have been shown also for other boundary conditions (see e.g., \cite{brock,weinberger,wein2}). As for the biharmonic operator, much less is known. Lord Rayleigh conjectured that the ball minimizes the fundamental tone of the clamped plate (i.e., the first positive eigenvalue of the biharmonic operator with Dirichlet boundary conditions) among open sets with the same measure. This has been proved by Nadirashvili \cite{nadir} for $N=2$, and soon generalized by Ashbaugh and Benguria \cite{ash} for $N=3$, while the general case remains an open problem (see also \cite{mohr,szego}). Regarding Neumann boundary conditions, Chasman \cite{chas1} proved that the ball is a maximizer for the fundamental tone. We refer to \cite{henry} for a general approach to domain perturbation problems (see also \cite{hale}), and to \cite{he} for a comprehensive discussion on eigenvalue shape optimization problems for elliptic operators. We also refer to \cite{bula2013,buosohinged} where the authors prove analyticity properties in the spirit of \cite{lala2004} for Dirichlet and intermediate boundary conditions respectively, and show that balls are critical domains for all the elementary symmetric functions of the eigenvalues.

Problem (\ref{Steklov-Bi}) should not be confused with other Steklov-type problems already discussed in the literature. For example, in \cite{bucurgazzola} the authors consider the following problem
\begin{equation*}
\left\{\begin{array}{ll}
\Delta^2u=0, & {\rm in }\ \Omega,\\
u=0, & {\rm on}\ \partial\Omega,\\
\Delta u=\lambda\frac{\partial u}{\partial\nu}, & {\rm on}\ \partial\Omega,
\end{array}\right.
\end{equation*}
which has a rather different nature. We note that, broadly speaking, one may refer to Steklov-type boundary conditions for those problems where a spectral parameter enters the boundary conditions.

The aim of the present paper is to discuss the Steklov problem (\ref{Steklov-Bi}) as the natural fourth order version of problem (\ref{Steklov-2}).
We derive problem (\ref{Steklov-Bi}) starting from a physical model and study the relationship with the Neumann eigenvalue problem for the biharmonic operator considered in \cite{chas1}.
Then we adapt the arguments used in \cite{lala2004,lalacri} in order to show real analyticity of the elementary symmetric functions of the eigenvalues of (\ref{Steklov-Bi}) and compute Hadamard-type formulas, which are used to prove that balls are critical domains. For completeness, we do the same also for the Neumann problem as stated in \cite{chas1}.
Finally, we study problem (\ref{Steklov-Bi}) when $\Omega$ is a ball and identify the fundamental tone and the corresponding modes (the eigenfunctions). By following a scheme similar to that used in \cite{weinberger} we prove that the ball is a maximizer for the first positive eigenvalue of problem (\ref{Steklov-Bi}) among all bounded domains of class $C^1$.

The paper is organized as follows. In Section \ref{sec:2} we  derive problem (\ref{Steklov-Bi}) providing a physical interpretation. In Section \ref{sec:3} we characterize the spectrum and show that problem (\ref{Steklov-Bi}) is strictly related to the Neumann eigenvalue problem as described in \cite{chas1}. As a bypass product, we provide a further phyisical justification of (\ref{Steklov-Bi}). In Section \ref{sec:4} we compute Hadamard-type formulas and prove that balls are critical domains under measure constraint for the elementary symmetric functions of the eigenvalues of (\ref{Steklov-Bi}) and of the corresponding Neumann problem (\ref{neumannweak}). In Section \ref{sec:5} we prove the isoperimetric inequality for the fundamental tone. Finally, in Section \ref{sec:6}, we provide some remarks on problems (\ref{Steklov-Bi}) and (\ref{neumannweak}) when $\tau=0$.


\section{Formulating the problem}
\label{sec:2}
In this section we provide a physical interpretation of problem (\ref{Steklov-Bi}) for $N=2$, which arises in the theory of linear elasticity, in particular in the study of transverse vibrations of a thin plate. Actually, in Sections \ref{sec:2} and \ref{sec:3} we will consider a slightly more general version of problem (\ref{Steklov-Bi}), namely
\begin{equation}\label{Steklov-Bi-dens}
\left\{\begin{array}{ll}
\Delta^2 u -\tau\Delta u=0 ,\ \ & {\rm in}\  \Omega,\\
\frac{\partial^2 u}{\partial \nu^2 }=0 ,\ \ & {\rm on}\  \partial \Omega,\\
\tau\frac{\partial u}{\partial \nu }-{\rm div}_{\partial\Omega}\big(D^2u.\nu\big)-\frac{\partial\Delta u}{\partial\nu}=\lambda\rho u ,\ \ & {\rm on}\  \partial \Omega,
\end{array}\right.
\end{equation}
where a positive weight $\rho\in L^{\infty}(\partial\Omega)$ appears in the boundary conditions. The weight $\rho$ has the meaning of a mass density. We shall always assume that $\tau$ is a fixed positive real number. We recall that the tangential divergence ${\rm div}_{\partial\Omega}F$ of a vector field $F$ is defined as ${\rm div}_{\partial\Omega}F={\rm div}F_{|_{\partial\Omega}}-\left(DF.\nu\right)\cdot\nu$, where $DF$ is the Jacobiam matrix of $F$.

As usual, we assume that the  mass is displaced in the middle plane of the plate parallel to its faces. When the body is at its equilibrium it covers a planar domain $\Omega$ with boundary $\partial\Omega$ in $\mathbb R ^2$. We describe the vertical deviation from the equilibrium during the vibration of each point $(x,y)$ of $\Omega$ at time $t$ by means of a function $v=v(x,y,t)$. We suppose that the whole mass of the plate is concentrated at the boundary with a density which we denote by $\rho(x,y)$. Moreover, we assume that $\rho(x,y)$ is bounded and positive on $\partial\Omega$. Under these assumptions, the total kinetic energy of the plate is given by

$$
T=\frac{1}{2}\int_{\partial\Omega}\rho\dot v^2d\sigma,
$$
where we denote by $\dot v$ the derivative of $v$ with respect to the time $t$, and by $d\sigma$ the surface measure on $\partial\Omega$. Now we obtain an expression for the potential energy of the plate. By following \cite[\S 10.8]{wein}, under the assumption that the strain potential energy at each point depends only on the strain configuration at that point and that the Poisson ratio of the material is zero, we have that the strain potential energy is given by

$$
V_s=\frac{1}{2}\int_{\Omega}\big(v_{xx}^2+v_{yy}^2+2v_{xy}^2\big)dxdy.
$$
Besides $V_s$, we have another term of the potential energy due to the lateral tension

$$
V_t=\frac{\tau}{2}\int_{\Omega}\big(v_x^2+v_y^2\big)dxdy,
$$
where $\tau>0$ is the ratio of lateral tension due to flexural rigidity. The Hamilton's integral of the system is given by

\begin{multline}\label{corr1}
\mathcal H=\int_{t_1}^{t_2}T-V_s-V_t\, dt\\=\frac{1}{2}\int_{t_1}^{t_2}\int_{\partial\Omega}\rho\dot v^2 d\sigma dt-\frac{1}{2}\int_{t_1}^{t_2}\int_{\Omega}\big(v_{xx}^2+v_{yy}^2+2v_{xy}^2\big)+\tau\big(v_x^2+v_y^2\big)dxdydt.
\end{multline}
According to Hamilton's Variational Principle, the actual motion of the system minimizes such integral. Let $v(x,y,t)$ be a minimizer for $\mathcal H$. By differentiating (\ref{corr1}) it follows that $v$ satisfies 
\begin{multline*}
-\int_{t_1}^{t_2}\int_{\partial\Omega}\eta\rho\ddot v d\sigma dt-\int_{t_1}^{t_2}\int_{\Omega}\eta\left(\Delta^2v-\tau\Delta v\right)dxdydt\\
-\int_{t_1}^{t_2}\int_{\partial\Omega}\frac{\partial\eta}{\partial\nu}\frac{\partial^2v}{\partial\nu^2}-\eta\left(\tau\frac{\partial v}{\partial\nu}-{\rm div}_{\partial\Omega}\left(D^2v.\nu\right)_{\partial\Omega}-\frac{\partial\Delta v}{\partial\nu}\right)d\sigma dt=0\nonumber,
\end{multline*}
for all $\eta\in C^2(\Omega\times[t_1,t_2])$ such that $\eta(x,y,t_1)=\eta(x,y,t_2)=0$ and  $(x,y)\in\Omega$. Since $\eta$ is arbitrary we obtain
\begin{equation}\label{forcond}
   \begin{cases}
	\Delta^2v-\tau\Delta v=0, & {\rm in}\ \Omega,\\
   \frac{\partial^2v}{\partial\nu^2}=0, & {\rm on}\ \partial\Omega,\\
	\rho\ddot v+\tau\frac{\partial v}{\partial\nu}-{\rm div}_{\partial\Omega}\left(D^2v.\nu\right)-\frac{\partial\Delta v}{\partial\nu}=0, & {\rm on}\ \partial\Omega
   \end{cases}
\end{equation}
for all $t\in[t_1,t_2]$. We remark that we wrote ${\rm div}_{\partial\Omega}\left(D^2v.\nu\right)$ instead of ${\rm div}_{\partial\Omega}\left(D^2v.\nu\right)_{\partial\Omega}$ since $\left(D^2v.\nu\right)_{\partial\Omega}=D^2v.\nu-\frac{\partial^2 v}{\partial\nu^2}\nu$ and $\frac{\partial^2 v}{\partial\nu^2}=0$ on $\partial\Omega$.

We separate the variables and, as is customary, we look for solutions to problem (\ref{forcond}) of the form $v(x,y,t)=u(x,y)w(t)$. We find that the temporal component $w(t)$ solves the ordinary differential equation $-\ddot w(t)=\lambda w(t)$ for all $t\in[t_1,t_2]$, while the spatial component $u$ solves problem (\ref{Steklov-Bi-dens}).


\section{Characterization of the spectrum. Alternative derivation of the problem}
\label{sec:3}

In this section we prove that the spectrum of the eigenvalue problem (\ref{Steklov-Bi-dens}) is discrete. In particular, each eigenvalue is non-negative and has finite multiplicity and there exists a Hilbert basis of the standard Sobolev space $H^2(\Omega)$ of eigenvectors. Then we provide a further derivation of problem (\ref{Steklov-Bi-dens}). Namely, we show that this problem can be seen as a limit of eigenvalue problems for the biharmonic operator with Neumann boundary conditions and mass density $\rho_{\varepsilon}$ which concentrates in a neighborhood of the boundary as $\varepsilon$ goes to zero. We refer to \cite{arrieta,proz} for similar discussions concerning second order problems. We observe that the asymptotic analysis of mass concentration problems for second order operators has been performed by several authors by exploiting asymptotic expansions methods, see e.g., \cite{nazarov1,nazarov2} and the references therein. We also mention the alternative approach based on potential theory and functional analysis proposed in \cite{dalla,lanza}.

Note that here and in the sequel we shall not put any restriction on the space dimension. Thus $\Omega$ will always denote a bounded domain in $\mathbb R^N$ of class $C^1$, with $N\geq 2$.

\subsection{Analysis of the spectrum of problem (\ref{Steklov-Bi-dens})}
Let $\rho\in{\mathcal R^{\mathcal S}}$, where ${\mathcal R^{\mathcal S}}:=\left\{\rho\in L^\infty(\partial\Omega):{\rm ess\,inf}_{x\in\partial\Omega}\rho(x)>0\right\}$. We consider the weak formulation of problem (\ref{Steklov-Bi-dens}),

\begin{equation}\label{steklovweak}
\int_\Omega D^2u:D^2\varphi+\tau\nabla u\cdot\nabla\varphi dx=\lambda\int_{\partial\Omega}\rho u \varphi d\sigma\,,\ \forall\varphi\in H^2(\Omega)\,,
\end{equation}
in the unknowns $u\in H^2(\Omega)$, $\lambda\in\mathbb{R}$, where
$$
D^2u:D^2\varphi=\sum_{i,j=1}^N\frac{\partial^2u}{\partial x_i\partial x_j}\frac{\partial^2 \varphi}{\partial x_i\partial x_j}
$$
denotes the Frobenius product. Actually, we will obtain a problem in $H^2(\Omega)/\mathbb{R}$ since we need to get rid of the constants, which generate  the eigenspace corresponding to the eigenvalue $\lambda=0$. We denote by $\mathcal J^{\mathcal S}_\rho$ the continuous embedding of $L^2(\partial\Omega)$ into $H^2(\Omega)'$ defined by
\begin{equation*}
 \mathcal J^{\mathcal S}_\rho[u][\varphi]:=\int_{\partial\Omega}\rho u\varphi d\sigma,\ \forall u\in L^2(\partial\Omega),\varphi\in H^2(\Omega).
\end{equation*}
We set
\begin{equation*}
 H^{2,\mathcal S}_\rho(\Omega):=\left\{u\in H^2(\Omega):\int_{\partial\Omega}\rho u d\sigma=0\right\},
\end{equation*}
and we consider in $H^2(\Omega)$ the bilinear form 
\begin{equation}\label{bilinear-form}
<u,v>=\int_\Omega D^2u:D^2v+\tau\nabla u\cdot\nabla v dx.
\end{equation}
By the Poincar\'e-Wirtinger Inequality, it turns out that this bilinear form is indeed a scalar product on ${{H}}^{2,\mathcal S}_\rho(\Omega)$ whose induced norm is equivalent to the standard one. In the sequel we will think of the space $H^{2,\mathcal S}_{\rho}(\Omega)$ as endowed with the form (\ref{bilinear-form}). Let $F(\Omega)$ be defined by $F(\Omega):=\left\{G\in H^2(\Omega)':G[1]=0\right\}$. Then, we consider the operator $\mathcal P^{\mathcal S}_\rho$ as an operator from ${{H}}^{2,\mathcal S}_\rho(\Omega)$ to $F(\Omega)$, defined by
\begin{equation}\label{steklovoperator}
\mathcal P^{\mathcal S}_\rho[u][\varphi]:=\int_\Omega  D^2u:D^2\varphi+\tau \nabla u\cdot\nabla\varphi dx,\ \forall u\in {{H}}^{2,\mathcal S}_\rho(\Omega),\varphi\in {H}^2(\Omega).
\end{equation}
It turns out that $\mathcal{\mathcal P}^{\mathcal S}_\rho$ is a homeomorphism of ${{H}}^{2,\mathcal S}_\rho(\Omega)$ onto $F(\Omega)$. We define the operator $\pi^{\mathcal S}_{\rho}$ from $H^2(\Omega)$ to $ H^{2,\mathcal S}_\rho(\Omega)$ by
\begin{equation}\label{piS}
\pi^{\mathcal S}_{\rho}[u]:=u-\frac{\int_{\partial\Omega}\rho u d\sigma}{\int_{\partial\Omega}\rho d\sigma}.
\end{equation}
We consider the space ${H}^2(\Omega)/\mathbb{R}$ endowed with the bilinear form induced by (\ref{bilinear-form}). Such bilinear form renders ${H}^2(\Omega)/\mathbb{R}$ a Hilbert space. We denote by $\pi^{\sharp,\mathcal S}_{\rho}$ the map from ${H}^{2}(\Omega)/\mathbb{R}$ onto ${{H}}^{2,\mathcal S}_\rho(\Omega)$ defined by the equality $\pi^{\mathcal S}_{\rho}=\pi^{\sharp,\mathcal S}_{\rho}\circ p$, where $p$ is the canonical projection of $H^2(\Omega)$ onto $H^2(\Omega)/\mathbb R$. The map $\pi^{\sharp,\mathcal S}_{\rho}$ turns out to be a homeomorphism. Finally, we define the operator $ T^{\mathcal S}_\rho$ acting on ${H}^2(\Omega)/\mathbb{R}$ as follows
\begin{equation}\label{stekres}
 T^{\mathcal S}_\rho:=(\pi^{\sharp,\mathcal S}_{\rho})^{-1}\circ({\mathcal P}^{\mathcal S}_\rho)^{-1}\circ {\mathcal J}^{\mathcal S}_\rho\circ {\rm Tr}\circ\pi^{\sharp,\mathcal S}_{\rho},
\end{equation}
where ${\rm Tr}$ denotes the trace operator acting from $H^2(\Omega)$ to $L^2(\partial\Omega)$.

\begin{rem}
We observe that the pair $(\lambda,u)$ of the set $(\mathbb{R}\setminus\lbrace 0\rbrace)\times ({{H}}^{2,\mathcal S}_\rho(\Omega)\setminus\lbrace 0\rbrace)$ satisfies (\ref{steklovweak}) if and only if $\lambda>0$ and the pair $(\lambda^{-1},p[u])$ of the set $\mathbb{R}\times (({H}^2(\Omega)/\mathbb{R})\setminus\lbrace 0\rbrace)$ satisfies the equation
\begin{equation*}
\lambda^{-1} p[u]= T^{\mathcal S}_\rho p[u].
\end{equation*}
\end{rem}

We have the following
\begin{thm}
\label{comsa}
The operator $ T^{\mathcal S}_\rho$ is a non-negative compact selfadjoint operator in ${H}^2(\Omega)/\mathbb{R}$, whose eigenvalues coincide with the reciprocals of the positive eigenvalues of problem (\ref{steklovweak}). In particular, the set of eigenvalues of problem (\ref{steklovweak}) is contained in $[0,+\infty[$ and consists of the image of a sequence increasing to $+\infty$. Each eigenvalue has finite multiplicity.
\proof For the selfadjointness, it suffices to observe that
\begin{eqnarray*}
<  T^{\mathcal S}_\rho u, v>_{{H}^2(\Omega)/\mathbb{R}}&=&<(\pi^{\sharp,\mathcal S}_{\rho})^{-1}\circ({\mathcal P}^{\mathcal S}_\rho)^{-1}\circ {\mathcal J}^{\mathcal S}_\rho\circ {\rm Tr}\circ\pi^{\sharp,\mathcal S}_{\rho} u, v>_{{H}^2(\Omega)/\mathbb{R}}\\&=&{\mathcal P}^{\mathcal S}_\rho[({\mathcal P}^{\mathcal S}_\rho)^{-1}\circ {\mathcal J}^{\mathcal S}_\rho\circ {\rm Tr} \circ \pi^{\sharp,\mathcal S}_{\rho} u][\pi^{\sharp,\mathcal S}_{\rho} v]\\
&=&{\mathcal J}^{\mathcal S}_\rho[{\rm Tr}\circ\pi^{\sharp,\mathcal S}_{\rho} u][\pi^{\sharp,\mathcal S}_{\rho} v],\ \ \forall u,v\in {H}^2(\Omega)/\mathbb{R}.
\end{eqnarray*}
For the compactness, just observe that the trace operator ${\rm Tr}$ acting from $H^1(\Omega)$ to $L^2(\partial\Omega)$ is compact. The remaining statements are straightforward.
\endproof

\end{thm}
As a consequence we have that the spectrum of (\ref{steklovweak}) is of the form
$$
0\leq\lambda_1\leq\lambda_2\leq\cdots\leq\lambda_j\leq\cdots
$$
Note that the first positive eigenvalue is $\lambda_{2}$ as proved by the following
\begin{thm}\label{comsa2}
The first eigenvalue $\lambda_1$ of (\ref{steklovweak}) is zero and the corresponding eigenfunctions are the constants. Moreover, $\lambda_{2}>0$.
\proof
It is straightforward to see that constant functions are eigenfunctions of (\ref{steklovweak}) with eigenvalue $\lambda=0$. Suppose now that $u$ is an eigenfunction corresponding to the eigenvalue $\lambda=0$. Then we have
$$
\int_{\Omega}|D^2u|^2+\tau |\nabla u|^2 dx=0,
$$
where $|D^2u|^2=\sum_{i,j=1}^N\big(\frac{\partial^2 u}{\partial x_i\partial x_j}\big)^2$. Since $\nabla u=0$, it follows that $u$ is constant. Then the eigenvalue $\lambda=0$ has multiplicity one. 
\endproof
\end{thm}

Thus $\lambda_{2}$ is the first positive eigenvalue of (\ref{steklovweak}) which is usually called the fundamental tone. Note that we can charactrize $\lambda_{2}$ by means of the Rayleigh principle
\begin{equation}\label{minmax}
\lambda_{2}=\min_{\substack{0\ne u\in H^2(\Omega)\\ \int_{\partial\Omega}\rho u d\sigma=0}}\frac{\int_{\Omega}|D^2u|^2+\tau |\nabla u|^2 dx}{\int_{\partial\Omega}\rho u^2 d\sigma}.
\end{equation}




\subsection{Asymptotic behavior of Neumann eigenvalues}

We consider the following eigenvalue problem for the biharmonic operator subject to Neumann boundary conditions

\begin{equation}\label{neumannweak}
\left\{\begin{array}{ll}
\Delta^2 u-\tau\Delta u = \lambda\rho u,\ \ & {\rm in}\  \Omega,\\
\frac{\partial^2 u}{\partial \nu^2 }=0 ,\ \ & {\rm on}\  \partial \Omega,\\
\tau\frac{\partial u}{\partial \nu }-{\rm div_{\partial\Omega}}\big(D^2u.\nu\big)-\frac{\partial\Delta u}{\partial\nu}=0 ,\ \ & {\rm on}\  \partial \Omega,
\end{array}\right.
\end{equation}
where $\rho\in {\mathcal R^{\mathcal N}}:=\left\{\rho\in L^\infty(\Omega):{\rm ess\,inf}_{x\in\Omega}\rho(x)>0\right\}$ (we refer to \cite{chas1} for the derivation of the boundary conditions).
It is well known that this problem arises in the study of a free vibrating plate whose mass is displaced on the whole of $\Omega$ with density $\rho$. 

Let us denote by $\Omega_{\varepsilon}$ the set defined by
$$
\Omega_{\varepsilon}:=\left\{x\in\Omega:{\rm dist}(x,\partial\Omega)>\varepsilon\right\}.
$$
We fix a positive number $M>0$ and  choose the family of densities $\rho_{\varepsilon}$ defined as follows

\begin{equation}\label{rhoeps}
   \rho_{\varepsilon}(x)=\begin{cases}
	\varepsilon, & {\rm if}\ x\in\Omega_{\varepsilon},\\
   \frac{M-\varepsilon |\Omega_{\varepsilon}|}{|\Omega\setminus\overline{\Omega}_{\varepsilon}|}, & {\rm if}\ x\in\Omega\setminus\overline{\Omega}_{\varepsilon},\\
   \end{cases}
\end{equation}
for $\varepsilon\in ]0,\varepsilon_0[$, where $\varepsilon_0>0$ is sufficiently small. If in addition we assume that $\Omega$ is of class $C^2$, $\varepsilon_0$ can be chosen in such a way that the map $x\mapsto x-\nu\varepsilon$ is a diffeomorphism between $\partial\Omega$ and $\partial\Omega_{\varepsilon}$ for all $\varepsilon\in]0,\varepsilon_0[$. 
We note that $\int_{\Omega}\rho_{\varepsilon}dx=M$ for all $\varepsilon\in]0,\varepsilon_0[$. We refer to the quantity $M$ as the total mass of the body.

We prove, under the additional hypothesis that $\Omega$ is of class $C^2$, convergence of the eigenvalues and eigenfunctions of problem (\ref{neumannweak}) with density $\rho_{\eps}$ to the eigenvalues and eigenfunctions of problem (\ref{steklovweak}) with constant surface density $\frac{M}{|\partial\Omega|}$ when the parameter $\varepsilon$ go to zero (see Corollary \ref{corr4}). This provides a further interpretation of problem (\ref{steklovweak}) as the equation of a free vibrating plate whose mass is concentrated at the boundary in the case of domains of class $C^2$.

Problem (\ref{neumannweak}) has an increasing sequence of non-negative eigenvalues of finite multiplicity and the eigenfunctions form a Hilbert basis of $H^2(\Omega)$. We consider the weak formulation of problem (\ref{neumannweak}) with density $\rho_{\varepsilon}$,

\begin{equation}\label{neumannweak2}
\int_\Omega  D^2u:D^2\varphi+\tau \nabla u\cdot\nabla\varphi dx=\lambda\int_\Omega\rho_{\varepsilon} u \varphi dx\,,\  \forall\varphi\in H^2(\Omega)\,,
\end{equation}
in the unknowns $u\in H^2(\Omega)$, $\lambda\in\mathbb{R}$. In the sequel we shall recast this problem in $H^2(\Omega)/\mathbb{R}$ since we need to get rid of the constants, which generate  the eigenspace corresponding to the eigenvalue $\lambda=0$. We denote by $i$ the canonical embedding of $H^2(\Omega)$ into $L^2(\Omega)$. We denote by $\mathcal J^{\mathcal N}_{\rho_{\varepsilon}}$ the continuous embedding of $L^2(\Omega)$ into $H^2(\Omega)'$, defined by
\begin{equation*}
\mathcal J^{\mathcal N}_{\rho_{\varepsilon}}[u][\varphi]:=\int_\Omega\rho_{\varepsilon} u\varphi dx,\ \forall u\in L^2(\Omega),\varphi\in H^2(\Omega).
\end{equation*}
We set
\begin{equation*}
H^{2,\mathcal N}_{\rho_{\varepsilon}}(\Omega):=\left\{u\in H^2(\Omega):\int_\Omega u \rho_{\varepsilon} dx=0\right\}.
\end{equation*}
In the sequel we will think of the space $H^{2,\mathcal N}_{\rho_{\varepsilon}}(\Omega)$ as endowed with the form (\ref{bilinear-form}). Such form defines on ${H}^{2,\mathcal N}_{\rho_{\varepsilon}}(\Omega)$ a scalar product whose induced norm is equivalent to the standard one. We denote by $\pi^{\mathcal N}_{\rho_{\varepsilon}}$ the map from ${H}^{2}(\Omega)$ to ${H}^{2,\mathcal N}_{\rho_{\varepsilon}}(\Omega)$ defined by
$$
\pi^{\mathcal N}_{\rho_{\varepsilon}}[u]:=u-\frac{\int_\Omega u\rho_{\varepsilon} dx}{\int_\Omega\rho_{\varepsilon} dx},
$$
for all $u\in {H}^{2}(\Omega)$. We denote by $\pi^{\sharp,\mathcal N}_{\rho_{\varepsilon}}$ the map from ${H}^{2}(\Omega)/\mathbb{R}$ onto ${H}^{2,\mathcal N}_{\rho_{\varepsilon}}(\Omega)$ defined by the equality $\pi^{\mathcal N}_{\rho_{\varepsilon}}=\pi^{\sharp,\mathcal N}_{\rho_{\varepsilon}}\circ p$. As in (\ref{steklovoperator}), we consider the operator $\mathcal P^{\mathcal N}_{\rho_{\varepsilon}}$ as a map from ${H}^{2,\mathcal N}_{\rho_{\varepsilon}}(\Omega)$ to $F(\Omega)$ defined by
\begin{equation*}
\mathcal P^{\mathcal N}_{\rho_{\varepsilon}}[u][\varphi]:=\int_\Omega D^2u:D^2\varphi+\tau\nabla u\cdot\nabla\varphi dx, \ \forall u\in {H}^{2,\mathcal N}_{\rho_{\varepsilon}}(\Omega),\varphi \in {H}^{2}(\Omega).
\end{equation*}
It turns out that $\mathcal P^{\mathcal N}_{\rho_{\varepsilon}}$ is a linear homeomorphism of ${H}^{2,\mathcal N}_{\rho_{\varepsilon}}(\Omega)$ onto $F(\Omega)$. Finally, let the operator $T^{\mathcal N}_{\rho_{\varepsilon}}$ from ${H}^2(\Omega)/\mathbb{R}$ to itself be defined by
\begin{equation}\label{Trho}
T^{\mathcal N}_{\rho_{\varepsilon}}:=(\pi^{\sharp,\mathcal N}_{\rho_{\varepsilon}})^{-1}\circ (\mathcal P^{\mathcal N}_{\rho_{\varepsilon}})^{-1}\circ \mathcal J^{\mathcal N}_{\rho_{\varepsilon}}\circ i \circ \pi^{\sharp,\mathcal N}_{\rho_{\varepsilon}}.
\end{equation}

\begin{rem}
We observe that the pair $(\lambda,u)$ of the set $(\mathbb{R}\setminus\lbrace 0\rbrace)\times ({H}^{2,\mathcal N}_{\rho_{\varepsilon}}(\Omega)\setminus\lbrace 0\rbrace)$ satisfies (\ref{neumannweak2}) if and only if $\lambda>0$ and the pair $(\lambda^{-1},p[u])$ of the set $\mathbb{R}\times (({H}^2(\Omega)/\mathbb{R})\setminus\lbrace 0\rbrace)$ satisfies the equation
\begin{equation*}
\lambda^{-1} p[u]=T^{\mathcal N}_{\rho_{\varepsilon}} p[u].
\end{equation*}
\end{rem}

As in Theorem \ref{comsa} it is easy to prove the following

\begin{thm}
Let $\Omega$ be a bounded domain in $\mathbb{R}^N$ of class $C^1$ and $\varepsilon\in]0,\varepsilon_0[$. The operator $T^{\mathcal N}_{\rho_{\varepsilon}}$ is a compact selfadjoint operator in ${H}^2(\Omega)/\mathbb{R}$ and its eigenvalues coincide with the reciprocals of the positive eigenvalues $\lambda_j({\rho_{\varepsilon}})$ of problem (\ref{neumannweak2}) for all $j\in\mathbb N$. Moreover, the set of eigenvalues of problem (\ref{neumannweak2}) is contained in $[0,+\infty[$ and consists of the image of a sequence increasing to $+\infty$. Each eigenvalue has finite multiplicity.
\end{thm}

We have the following theorem on the spectrum of problem (\ref{neumannweak2}) (see also Theorem \ref{comsa2}).
\begin{thm}
The first eigenvalue $\lambda_1({\rho_{\varepsilon}})$ of (\ref{neumannweak2}) is zero and the corresponding eigenfunctions are the constants. Moreover, $\lambda_2({\rho_{\varepsilon}})>0$.\end{thm}

Now we highlight the relations between problems (\ref{steklovweak}) and (\ref{neumannweak2}) when $\Omega$ is of class $C^2$. In particular we plan to prove the following 

\begin{thm}\label{convergenzacompattaultimo}
Let $\Omega$ be a bounded domain in $\mathbb{R}^N$ of class $C^2$. Let the operators $ T^{\mathcal S}_{\frac{M}{|\partial\Omega|}}$ and $  T^{\mathcal N}_{\rho_{\varepsilon}}$ from $H^2(\Omega)/\mathbb R$ to itself be defined as in (\ref{stekres}) and (\ref{Trho}) respectively. Then the sequence $\lbrace  T^{\mathcal N}_{\rho_{\varepsilon}}\rbrace_{\varepsilon\in]0,\varepsilon_0[}$ converges in norm to $ T^{\mathcal S}_{\frac{M}{|\partial\Omega|}}$ as $\varepsilon\rightarrow 0$.
\end{thm}

We need some preliminary results in order to prove Theorem \ref{convergenzacompattaultimo}. We remark that $\pi^{\sharp,\mathcal S}_c=\pi^{\sharp,\mathcal S}_1$ for all $c\in\mathbb R$, with $c\ne 0$. This can be easily deduced by (\ref{piS}).

\begin{lem}\label{lemma_neumann_to_steklov}
Let $\Omega$ be a bounded domain in $\mathbb{R}^N$ of class $C^2$. Let $\rho_{\varepsilon}\in\mathcal R^{\mathcal N}$ be as in (\ref{rhoeps}). Then the following statements hold.
\begin{enumerate}[i)]
\item {For all $\varphi\in {H}^2(\Omega)/\mathbb{R}$, $\pi^{\sharp,\mathcal N}_{\rho_{\varepsilon}}[\varphi]\rightarrow\pi^{\sharp,\mathcal S}_1[\varphi]$ in $L^2(\Omega)$ {\rm (}hence also in ${H}^2(\Omega)${\rm )} as $\varepsilon\rightarrow 0$;}
\item {If $u_{\varepsilon}\rightharpoonup u$ in ${H}^2(\Omega)/\mathbb{R}$, then (possibly passing to a subsequence) $\pi^{\sharp,\mathcal N}_{\rho_{\varepsilon}}[u_{\varepsilon}]\rightarrow\pi^{\sharp,\mathcal S}_1[u]$ in $L^2(\Omega)$ as $\varepsilon\rightarrow 0$;}
\item {Assume that $u_{\eps}, u, w_{\eps}, w\in H^2(\Omega)$ are such that $u_{\varepsilon}\rightarrow u$, $w_{\varepsilon}\rightarrow w$ in $L^2(\Omega)$, ${\rm Tr}[u_{\varepsilon}]\rightarrow {\rm Tr}[u]$, ${\rm Tr} [w_{\varepsilon}]\rightarrow {\rm Tr}[w]$ in $L^2(\partial \Omega)$ as $\varepsilon\rightarrow 0$. Moreover assume that there exists a constant $C>0$ such that $\norm{\nabla u_{\varepsilon}}_{L^2(\Omega)}\leq C$, $\norm{\nabla w_{\varepsilon}}_{L^2(\Omega)}\leq C$ for all $\varepsilon\in]0,\varepsilon_0[$. Then
    \begin{eqnarray*}
    &&\int_{\Omega}\rho_{\varepsilon}\left(u_{\varepsilon}-u\right)w_{\varepsilon} dx\rightarrow 0
    \end{eqnarray*}
    and
    \begin{eqnarray*}
    &&\int_{\Omega}\rho_{\varepsilon}\left(w_{\varepsilon}-w\right)u dx\rightarrow 0
    \end{eqnarray*}
		as $\varepsilon\rightarrow 0$}
\end{enumerate}
\proof
The proof is standard but long and we omit it. We refer to \cite{proz} and references therein for details.
We remark that in order to prove this lemma we need to assume that $\Omega$ is of class $C^2$, since the Tubular Neighborhood Theorem is used to perform computations on the strip $\Omega\setminus\overline{\Omega}_{\varepsilon}$.
\endproof
\end{lem}


\begin{proof}[Proof of Theorem \ref{convergenzacompattaultimo}]
It is sufficient to prove that the family $\left\{T^{\mathcal N}_{\rho_{\varepsilon}}\right\}_{\eps\in]0,\eps_0[}$ of compact operators, compactly converges to the compact operator $ T^{\mathcal S}_{\frac{M}{|\partial\Omega|}}$. This implies, in fact, that 
\begin{equation}\label{squareconv}
\lim_{{\varepsilon}\rightarrow 0}\norm{\left(T^{\mathcal N}_{\rho_{\varepsilon}}- T^{\mathcal S}_{\frac{M}{|\partial\Omega|}}\right)^2}_{\mathcal L(H^2(\Omega)/\mathbb R,H^2(\Omega)/\mathbb R)}=0.
\end{equation}
Then, since the operators $\left\{T^{\mathcal N}_{\rho_{\varepsilon}}\right\}_{\eps\in]0,\eps_0[}$ and $ T^{\mathcal S}_{\frac{M}{|\partial\Omega|}}$ are selfadjoint, property (\ref{squareconv}) is equivalent to convergence in norm. We refer to \cite{anselone,wei} for a more detailed discussion on compact convergence of compact operators on Hilbert spaces. 
We recall that, by definition, $ T^{\mathcal N}_{\rho_{\varepsilon}}$ compactly converges to $ T^{\mathcal S}_{\frac{M}{|\partial\Omega|}}$ if the following requirements are fulfilled:
\begin{enumerate}[i)]
\item if $\norm{u_{\varepsilon}}_{{{H}^2(\Omega)/\rea}}\leq C$ for all $\eps\in]0,\eps_0[ $, then the family $\lbrace T^{\mathcal N}_{\rho_{\eps}} u_{\eps}\rbrace_{\eps\in]0,\eps_0[ }$ is relatively compact in ${{H}^2(\Omega)/\rea}$;
\item if $u_{\eps}\rightarrow u$ in ${{H}^2(\Omega)/\rea}$, then $ T^{\mathcal N}_{\rho_{\eps}} u_{\eps}\rightarrow  T^{\mathcal S}_{\frac{M}{|\partial\Omega|}} u$ in ${{H}^2(\Omega)/\rea}$.
\end{enumerate}
We prove i) first. For a fixed $u\in {{H}^2(\Omega)/\rea}$ we have
\begin{multline}\label{firstrow}
\lim_{\eps\rightarrow 0}\int_{\Omega} \rho_{\eps} \pi^{\sharp,\mathcal N}_{\rho_{\eps}}[u] dx = \lim_{\eps\rightarrow 0}\int_\Omega\rho_{\eps}\left(\pi^{\sharp,\mathcal N}_{\rho_{\eps}}[u]-\pi^{\sharp,\mathcal S}_1[u]\right)dx\\
+\left(\lim_{{\eps}\rightarrow 0}\int_{\Omega}\rho_{\eps}\pi^{\sharp,\mathcal S}_1[u]dx-\frac{M}{|\partial\Omega|}\int_{\partial \Omega}\pi^{\sharp,\mathcal S}_1[u]d\sigma\right)\\
+\frac{M}{|\partial\Omega|}\int_{\partial \Omega}\pi^{\sharp,\mathcal S}_1[u]d\sigma.
\end{multline}
By Lemma \ref{lemma_neumann_to_steklov} we have that the first summand in the right-hand side of (\ref{firstrow}) goes to zero as ${\eps}\rightarrow 0$ and by standard calculus we have that the second term goes to zero as ${\eps}\rightarrow 0$. Moreover, the equality $(\pi^{\sharp,\mathcal N}_{\rho_{\eps}})^{-1}\circ(\mathcal P^{\mathcal N}_{\rho_{\eps}})^{-1}=(\pi^{\sharp,\mathcal S}_{1})^{-1}\circ(\mathcal P^{\mathcal S}_{1})^{-1}$ holds. Therefore, it follows that $ T^{\mathcal N}_{\rho_{\eps}} u$ is bounded for each $u\in {{H}^2(\Omega)/\rea}$. Thus, by Banach-Steinhaus Theorem, there exists $C'$ such that $\norm{ T^{\mathcal N}_{\rho_{\eps}}}_{\mathcal L({{H}^2(\Omega)/\rea},{{H}^2(\Omega)/\rea})}\leq C'$ for all ${\eps}\in ]0,\varepsilon_0[$. Moreover, since $\norm{u_{\eps}}_{{{H}^2(\Omega)/\rea}}\leq C$ for all ${\eps}\in]0,\eps_0[$, possibly passing to a subsequence, we have that $u_{\eps}\rightharpoonup u$ in ${{H}^2(\Omega)/\rea}$, for some $u\in H^2(\Omega)/\mathbb R$. This implies that, possibly passing to a subsequence, $ T^{\mathcal N}_{\rho_{\eps}} u_{\eps}\rightharpoonup w$ in ${{H}^2(\Omega)/\rea}$ as ${\eps}\rightarrow 0$, for some $w\in H^2(\Omega)/\mathbb R$. We show that $w= T^{\mathcal S}_{\frac{M}{|\partial\Omega|}}u$. To shorten our notation we set $w_{\eps}:= T^{\mathcal N}_{\rho_{\eps}} u_{\eps}$. By Lemma \ref{lemma_neumann_to_steklov} we have
\begin{eqnarray*}
\lim_{{\eps}\rightarrow 0}\int_{\Omega}D^2(\pi^{\sharp,\mathcal N}_{\rho_{\eps}}[w_{\eps}]):D^2(\pi^{\sharp,\mathcal N}_{\rho_{\eps}}[\varphi])+\tau\nabla(\pi^{\sharp,\mathcal N}_{\rho_{\eps}}[w_{\eps}])\cdot\nabla(\pi^{\sharp,\mathcal N}_{\rho_{\eps}}[\varphi])dx\\
=\int_{\Omega}D^2(\pi^{\sharp,\mathcal S}_{1}[w]):D^2(\pi^{\sharp,\mathcal S}_{1}[\varphi])+\tau\nabla(\pi^{\sharp,\mathcal S}_{1}[w])\cdot\nabla(\pi^{\sharp,\mathcal S}_{1}[\varphi])dx,
\end{eqnarray*}
for all $\varphi\in {{H}^2(\Omega)/\rea}$. On the other hand, since $\left(\mathcal P^{\mathcal N}_{\rho_{\eps}}\circ\pi^{\sharp,\mathcal N}_{\rho_{\eps}}\right)w_{\eps}=\left(\mathcal J^{\mathcal N}_{\rho_{\eps}}\circ i\circ \pi^{\sharp,\mathcal N}_{\rho_{\eps}}\right)u_{\eps}$, we have that
\begin{eqnarray*}
\int_{\Omega}D^2(\pi^{\sharp,\mathcal N}_{\rho_{\eps}}[w_{\eps}]):D^2(\pi^{\sharp,\mathcal N}_{\rho_{\eps}}[\varphi])+\tau\nabla(\pi^{\sharp,\mathcal N}_{\rho_{\eps}}[w_{\eps}])\cdot\nabla(\pi^{\sharp,\mathcal N}_{\rho_{\eps}}[\varphi])dx=\int_{\Omega}\rho_{\eps}\pi^{\sharp,\mathcal N}_{\rho_{\eps}}[u_{\eps}]\pi^{\sharp,\mathcal N}_{\rho_{\eps}}[\varphi]dx
\end{eqnarray*}
Then, by Lemma \ref{lemma_neumann_to_steklov}, $iii)$ we have
\begin{multline*}
<w,\varphi>_{{{H}^2(\Omega)/\rea}}
=\lim_{{\eps}\rightarrow 0}<w_{\eps},\varphi>_{{{H}^2(\Omega)/\rea}}
=\lim_{{\eps}\rightarrow 0}\int_{\Omega}\rho_{{\eps}}\pi^{\sharp,\mathcal N}_{\rho_{\eps}}[u_{\eps}]\pi^{\sharp,\mathcal N}_{\rho_{\eps}}[\varphi]dx\\
=\lim_{{\eps}\rightarrow 0}\int_{\Omega}\rho_{\eps}\left(\pi^{\sharp,\mathcal N}_{\rho_{\eps}}[u_{\eps}]-\pi^{\sharp,\mathcal S}_1[u]\right)\pi^{\sharp,\mathcal N}_{\rho_{\eps}}[\varphi]dx
+\lim_{{\eps}\rightarrow 0}\int_{\Omega}\rho_{\eps}\pi^{\sharp,\mathcal S}_1[u]\left(\pi^{\sharp,\mathcal N}_{\rho_{\eps}}[\varphi]-\pi^{\sharp,\mathcal S}_1[\varphi]\right)dx\\
+\lim_{{\eps}\rightarrow 0}\int_\Omega\rho_{\eps}\pi^{\sharp,\mathcal S}_1[u]\pi^{\sharp,\mathcal S}_1[\varphi]dx
=\frac{M}{|\partial\Omega|}\int_{\partial \Omega}\pi^{\sharp,\mathcal S}_1[u]\pi^{\sharp,\mathcal S}_1[\varphi]d\sigma
=< T^{\mathcal S}_{\frac{M}{|\partial\Omega|}}u,\varphi>_{{{H}^2(\Omega)/\rea}},
\end{multline*}
hence $w= T^{\mathcal S}_{\frac{M}{|\partial\Omega|}}u$. In a similar way one can prove that $\norm{w_{\eps}}_{{{H}^2(\Omega)/\rea}}\rightarrow\norm{w}_{{{H}^2(\Omega)/\rea}}$. In fact
\begin{multline*}
\lim_{{\eps}\rightarrow 0}\norm{w_{\varepsilon}}^2_{{{H}^2(\Omega)/\rea}}
=\lim_{{\eps}\rightarrow 0}\int_{\Omega}\rho_{\eps}\left(\pi^{\sharp,\mathcal N}_{\rho_{\eps}}[u_{\eps}]-\pi^{\sharp,\mathcal S}_1[u]\right)\pi^{\sharp,\mathcal N}_{\rho_{\eps}}[w_{\eps}]dx\\+\lim_{{\eps}\rightarrow 0}\int_{\Omega}\rho_{\eps}\pi^{\sharp,\mathcal S}_1[u]\left(\pi^{\sharp,\mathcal N}_{\rho_{\eps}}[w_{\eps}]-\pi^{\sharp,\mathcal S}_1[w_{\eps}]\right)dx\\+\lim_{{\eps}\rightarrow 0}\int_{\Omega}\rho_{\eps}\pi^{\sharp,\mathcal S}_1[u]\left(\pi^{\sharp,\mathcal S}_1[w_{\eps}]-\pi^{\sharp,\mathcal S}_1[w]\right)dx\\
+\lim_{{\eps}\rightarrow 0}\int_{\Omega}\rho_{\eps}\pi^{\sharp,\mathcal S}_1[u]\pi^{\sharp,\mathcal S}_1[w]dx\\
=\frac{M}{|\partial\Omega|}\int_{\partial \Omega}\pi^{\sharp,\mathcal S}_1[u]\pi^{\sharp,\mathcal S}_1[w]d\sigma=\norm{w}^2_{{{H}^2(\Omega)/\rea}}.
\end{multline*}
This proves $i)$. As for point $ii)$, let $u_{\eps}\rightarrow u$ in ${{H}^2(\Omega)/\rea}$. Then there exists $C''$ such that $\norm{u_{\eps}}_{{{H}^2(\Omega)/\rea}}\leq C''$ for all ${\eps}$. Then, by the same argument used for point $i)$, for each sequence ${\eps}_j\rightarrow 0$, possibly passing to a subsequence, we have $ T^{\mathcal N}_{\rho_{{\eps}_j}} u_{{\eps}_j}\rightarrow  T^{\mathcal S}_{\frac{M}{|\partial\Omega|}} u$. Since this is true  for each $\lbrace {\eps}_j\rbrace_{j\in\mathbb N}$, we have the convergence for the whole family, i.e., $T^{\mathcal N}_{\rho_{\eps}} u_{\eps}\rightarrow  T^{\mathcal S}_{\frac{M}{|\partial\Omega|}} {u}$. This concludes the proof.
\end{proof}

Now we recall the following well-known result.
\begin{thm}\label{convergselfadjoint}
Let $H$ be a real Hilbert space and $\left\{A_{\varepsilon}\right\}_{{\varepsilon}\in]0,{\varepsilon_0}[}$ a family of bounded selfadjoint operators converging in norm to the bounded selfadjoint operator $A$, i.e., $\lim_{{\varepsilon}\rightarrow 0}\norm{A_{\varepsilon}-A}_{\mathcal L(H,H)}$ $=0$. Then isolated eigenvalues $\lambda$ of $A$ of finite multiplicity are exactly the limits of eigenvalues of $A_{\varepsilon}$, counting multiplicity; moreover, the corresponding eigenprojections converge in norm.
\end{thm}

Thanks to Theorem \ref{convergselfadjoint}, as an immediate corollary of Theorem \ref{convergenzacompattaultimo} we have
\begin{corol}\label{corr4}
Let $\Omega$ be a bounded domain in $\mathbb{R}^N$ of class $C^2$. Let $\lambda_j(\rho_{\varepsilon})$ be the eigenvalues of problem (\ref{neumannweak2}) on $\Omega$ for all $j\in\mathbb N$. Let $\lambda_j$, $j\in\mathbb N$ denote the eigenvalues of problem (\ref{steklovweak}) corresponding to the constant surface density $\frac{M}{|\partial\Omega|}$. Then $\lim_{\varepsilon\rightarrow 0}\lambda_j(\rho_{\varepsilon})=\lambda_j$ for all $j\in\mathbb N$.
\end{corol}


\section{Symmetric functions of the eigenvalues. Isovolumetric perturbations}
\label{sec:4}

In this section we compute Hadamard-type formulas for both the Steklov and the Neumann problems, which will be used to investigate the behavior of the eigenvalues subject to isovolumetric perturbations. To do so, we use the so called transplantation method, see \cite{henry} for a general introduction to this approach.
We will study problems (\ref{Steklov-Bi}) and (\ref{neumannweak}) in $\phi(\Omega)$, for a suitable homeomorphism $\phi$, where $\Omega$ has to be thought as a fixed bounded domain of class $C^1$. Therefore, we	introduce the following class of functions
$$
\Phi(\Omega)=\left\{\phi\in\left(C^2\left(\overline{\Omega}\right)\right)^N:\phi\ {\rm injective}\ {\rm and}\ \inf_{\Omega}|\det D\phi|>0\right\}.
$$
We observe that if $\Omega$ is of class $C^1$ and $\phi\in\Phi(\Omega)$, then also $\phi(\Omega)$ is of class $C^1$ and $\phi^{(-1)}\in\Phi(\phi(\Omega))$. Therefore, it makes sense to study both problem (\ref{Steklov-Bi}) and problem (\ref{neumannweak}) on $\phi(\Omega)$. Moreover, we endow the space $C^2(\overline{\Omega})$ with the standard norm
$$
||f||_{C^2(\overline{\Omega})}=\sup_{|\alpha|\le 2,\ x\in\overline{\Omega}}|D^{\alpha}f(x)|.
$$
Note that $\Phi(\Omega)$ is open in $\left(C^2(\overline{\Omega})\right)^N$ (see \cite[Lemma 3.11]{lala2004}).

We recall that it has been pointed out that balls play a relevant role in the study of isovolumetric perturbations of the domain $\Omega$ for all the eigenvalues of the Dirichlet and Neumann Laplacian. We refer to \cite{lala2004,lala2007}, where the authors prove that the elementary symmetric functions of the eigenvalues depend real analytically on the domain, providing also Hadamard-type formulas for the corresponding derivatives. Then, in \cite{lalacri} they show that balls are critical points for such functions  under volume constraint.

From now on we will consider problems (\ref{Steklov-Bi}) and (\ref{neumannweak}) with constant mass density $\rho\equiv 1$.

\subsection{The Steklov problem}

We plan to study the Steklov problem in the domain $\phi(\Omega)$ for $\phi\in\Phi(\Omega)$, i.e.,
\begin{equation}
\label{steklovphi}
\left\{
\begin{array}{ll}
\Delta^2u-\tau\Delta u=0, & {\rm in\ } \phi(\Omega),\\
\frac{\partial^2u}{\partial\nu^2}=0, & {\rm on\ }\partial\phi(\Omega),\\
\tau\frac{\partial u}{\partial\nu}-\mathrm{div}_{\partial\phi(\Omega)}(D^2u.\nu)-\frac{\partial\Delta u}{\partial\nu}=\lambda u,&{\rm on\ }\partial\phi(\Omega).
\end{array}\right.
\end{equation}
To do so, we pull it back to $\Omega$. Therefore, we are interested in the operator $\mathcal P^{\mathcal S}_{\phi}$ from $ H^{2,\mathcal S}_{\phi}(\Omega)$ to $F(\Omega)$,
	defined by
	\begin{multline}
	\label{deltaphi}
	\mathcal P^{\mathcal S}_{\phi}[u][\varphi]:=\int_{\Omega}(D^2(u\circ\phi^{-1})\circ\phi):(D^2(\varphi\circ\phi^{-1})\circ\phi)|\det D\phi|dx\\
		+\tau\int_{\Omega}(\nabla(u\circ\phi^{-1})\circ\phi)\cdot(\nabla(\varphi\circ\phi^{-1})\circ\phi)|\det D\phi|dx,\ 
		\forall u\in H^{2,\mathcal S}_{\phi}(\Omega), \varphi\in H^2(\Omega),
	\end{multline}
	where
	$$
	 H^{2,\mathcal S}_{\phi}(\Omega):=\left\{u\in H^2(\Omega):\int_{\partial\Omega}u|\nu(\nabla\phi)^{-1}||\det D\phi|d\sigma=0\right\}.
	$$
Moreover, for every $\phi\in\Phi(\Omega)$, we consider the map ${\mathcal J}^{\mathcal S}_{\phi}$ from $L^2(\partial\Omega)$ to $H^2(\Omega)'$ defined by
	\begin{equation*}
	{\mathcal J}^{\mathcal S}_{\phi}[u][\varphi]:=\int_{\partial\Omega}u\varphi|\nu(\nabla\phi)^{-1}||\det D\phi|d\sigma,\ \forall u\in L^2(\partial\Omega),\varphi\in H^2(\Omega).
	\end{equation*}	
It is easily seen that the form (\ref{deltaphi}) is a scalar product on $ H^{2,\mathcal S}_{\phi}(\Omega)$. We will think of the space ${H}^{2,\mathcal S}_{\phi}(\Omega)$ as endowed with the scalar product (\ref{deltaphi}). We denote by ${\pi}^{\mathcal S}_{\phi}$  the map from $H^2(\Omega)$ to $ H^{2,\mathcal S}_{\phi}(\Omega)$ defined by
	$$
	\pi^{\mathcal S}_{\phi}(u)=:u-\frac{\int_{\partial\Omega}u|\nu(\nabla\phi)^{-1}||\det D\phi|dx}{\int_{\partial\Omega}|\nu(\nabla\phi)^{-1}||\det D\phi|dx},
	$$
	and by $\pi_{\phi}^{\sharp,\mathcal S}$ the map from $H^2(\Omega)/\mathbb{R}$ onto $ H^{2,\mathcal S}_{\phi}(\Omega)$ defined by the equality $\pi^{\mathcal S}_{\phi}=\pi_{\phi}^{\sharp,\mathcal S}\circ p$.
	Clearly, $\pi_{\phi}^{\sharp,\mathcal S}$ is a homeomorphism, and we can recast problem (\ref{steklovphi}) as
$$
\lambda^{-1}u=W^{\mathcal S}_{\phi}u,
$$
where
$$
W^{\mathcal S}_{\phi}:=(\pi_{\phi}^{\sharp,\mathcal S})^{-1}\circ(\mathcal P^{\mathcal S}_{\phi})^{-1}\circ{\mathcal J}^{\mathcal S}_{\phi}\circ\mathrm{Tr}\circ\pi_{\phi}^{\sharp,\mathcal S}.
$$
The operator $W^{\mathcal S}_{\phi}$ can be shown to be compact and selfadjoint, as we have done for the operator $ T^{\mathcal S}_{\rho}$ defined by (\ref{stekres}) in Theorem \ref{comsa} (see also \cite[Theorem 2.1]{lala2007}).

In order to avoid bifurcation phenomena, which usually occur when dealing with multiple eigenvalues, we turn our attention to the elementary symmetric functions of the eigenvalues. This is the aim of the following
	
	\begin{thm}
	\label{thesame}
	Let $\Omega$ be a bounded domain of $\mathbb{R}^N$ of class $C^1$.
	Let $F$ be a finite non-empty subset of $\mathbb{N}$.
	Let
	$$\mathcal{A}_{\Omega}[F]:=\{\phi\in\Phi(\Omega):\lambda_l[\phi]\notin\{\lambda_j[\phi]:j\in F\}\ \forall l\in\mathbb{N}
				\setminus F\}.$$
	Then the following statements hold.
	\begin{enumerate}[i)]
	\item The set $\mathcal{A}_{\Omega}[F]$ is open in $\Phi(\Omega)$. The map $P_F$ of $\mathcal{A}_{\Omega}[F]$ to the space
		$\mathcal{L}\left(H^2(\Omega),H^2(\Omega)\right)$ which takes $\phi\in\mathcal{A}_{\Omega}[F]$ to the orthogonal projection
		of $H^{2,\mathcal S}_{\phi}(\Omega)$ onto its (finite dimensional) subspace generated by
		$$\left\{u\in  H^{2,\mathcal S}_{\phi}(\Omega):\mathcal P_{\rm Id}^{\mathcal S}[u\circ\phi^{-1}]=\lambda_j[\phi]{\mathcal J}^{\mathcal S}_{\rm Id}\circ\mathrm{Tr}[u\circ\phi^{-1}]\ \mathrm{for\ some}\ j\in F\right\}$$
		is real analytic.
	\item Let $s\in\{1,\dots,|F|\}$. The function $\Lambda_{F,s}$ from $\mathcal{A}_{\Omega}[F]$ to $\mathbb{R}^N$ defined by
		$$\Lambda_{F,s}[\phi]:=\sum_{j_1<\dots<j_s\in F}\lambda_{j_1}[\phi]\cdots\lambda_{j_s}[\phi]$$
		is real analytic.
	\item Let
	$$\Theta_{\Omega}[F]:=\{\phi\in\mathcal{A}_{\Omega}[F]:\lambda_j[\phi]\ \text{have a common value}\ \lambda_F[\phi]\ \forall j\in F\}.$$
	Then the real analytic functions
	$$\left(\binom{|F|}{1}^{-1}\Lambda_{F,1}\right)^{\frac{1}{1}},\dots ,\left(\binom{|F|}{|F|}^{-1}\Lambda_{F,|F|}\right)^{\frac{1}{|F|}},$$
	of $\mathcal{A}_{\Omega}[F]$ to $\mathbb{R}^N$ coincide on $\Theta_{\Omega}[F]$ with the function which takes $\phi$ to $\lambda_F[\phi]$.
	\end{enumerate} 
	
	\proof
	The proof can be done adapting that of \cite[Theorem 2.2 and Corollary 2.3]{lala2007} (see also \cite{lala2004}).
	\endproof
	\end{thm}

In order to compute explicit formulas for the differentials of the functions $\Lambda_{F,s}$, we need the following technical lemma.

	\begin{lem}
	Let $\Omega$ be a bounded domain in $\mathbb{R}^N$ of class $C^1$, and let $\tilde\phi\in\Phi(\Omega)$ be such that $\tilde\phi(\Omega)$ is of class $C^2$.
	Let $u_1,u_2\in H^2(\Omega)$ be such that $v_i=u_i\circ\tilde{\phi}^{(-1)}\in H^4(\tilde{\phi}(\Omega))$ for $i=1,2$ and
	$$\frac{\partial^2v_1}{\partial\nu^2}=\frac{\partial^2v_2}{\partial\nu^2}=0\ \ \ {\rm on}\ \ \ \partial\tilde{\phi}(\Omega).$$
	Then we have
	\begin{multline}
	\label{formulafondamentale}
	d|_{\phi=\tilde{\phi}}\mathcal P^{\mathcal S}_{\phi}[\psi][u_1][u_2]=
	\int_{\partial\tilde{\phi}(\Omega)}(D^2v_1:D^2v_2+\tau\nabla v_1\cdot\nabla v_2)\mu\cdot\nu d\sigma\\
+\int_{\partial\tilde{\phi}(\Omega)}\left(\mathrm{div}_{\partial\tilde{\phi}(\Omega)}(D^2v_1.\nu)\nabla v_2+\mathrm{div}_{\partial\tilde{\phi}(\Omega)}(D^2v_2.\nu)\nabla v_1\right)\cdot\mu d\sigma\\
+\int_{\partial\tilde{\phi}(\Omega)}\left(\frac{\partial\Delta v_1}{\partial\nu}\nabla v_2+\frac{\partial\Delta v_2}{\partial\nu}\nabla v_1\right)\cdot\mu d\sigma
-\tau\int_{\partial\tilde{\phi}(\Omega)}\left(\frac{\partial v_1}{\partial\nu}\nabla v_2+\frac{\partial v_2}{\partial\nu}\nabla v_1\right)\cdot\mu d\sigma\\
-\int_{\tilde{\phi}(\Omega)}\left((\Delta^2v_1-\tau\Delta v_1)\nabla v_2+(\Delta^2v_2-\tau\Delta v_2)\nabla v_1\right)\cdot\mu d\sigma,
	\end{multline}
	for all $\psi\in (C^2(\overline{\Omega}))^N$, where $\mu=\psi\circ\tilde{\phi}^{-1}$. 

	\proof
	We have
	\begin{multline}
	\label{111}
		d|_{\phi=\tilde{\phi}}\mathcal P^{\mathcal S}_{\phi}[\psi][u_1][u_2]\\
			 =\int_{\Omega}(d|_{\phi=\tilde{\phi}}D^2(u_1\circ\phi^{-1})\circ\phi)[\psi]:
				(D^2(u_2\circ\tilde{\phi}^{-1})\circ\tilde{\phi})|\det D\tilde{\phi}|dx\  \\
				+\tau\int_{\Omega}(d|_{\phi=\tilde{\phi}}\nabla(u_1\circ\phi^{-1})\circ\phi)[\psi]\cdot
				(\nabla(u_2\circ\tilde{\phi}^{-1})\circ\tilde{\phi})|\det D\tilde{\phi}|dx\  \\
			 +\int_{\Omega}(D^2(u_1\circ\tilde{\phi}^{-1})\circ\tilde{\phi}):
				(d|_{\phi=\tilde{\phi}}D^2(u_2\circ\phi^{-1})\circ\phi)[\psi]|\det D\tilde{\phi}|dx\\
			 +\tau\int_{\Omega}(\nabla(u_1\circ\tilde{\phi}^{-1})\circ\tilde{\phi})\cdot
				(d|_{\phi=\tilde{\phi}}\nabla(u_2\circ\phi^{-1})\circ\phi)[\psi]|\det D\tilde{\phi}|dx\\
			 +\int_{\Omega}(D^2(u_1\circ\tilde{\phi}^{-1})\circ\tilde{\phi}):
				(D^2(u_2\circ\tilde{\phi}^{-1})\circ\tilde{\phi})d|_{\phi=\tilde{\phi}}|\det D\phi|[\psi]dx\\
				+\tau\int_{\Omega}(\nabla(u_1\circ\tilde{\phi}^{-1})\circ\tilde{\phi})\cdot
				(\nabla(u_2\circ\tilde{\phi}^{-1})\circ\tilde{\phi})d|_{\phi=\tilde{\phi}}|\det D\phi|[\psi]dx,
	\end{multline}
		and we note that the last two summands in (\ref{111}) equals
	$$\int_{\tilde{\phi}(\Omega)}\left(D^2v_1:D^2v_2+\tau\nabla v_1\cdot\nabla v_2\right)\mathrm{div}\mu dy.$$
	(See also Proprosition \ref{vfiprop}). By standard calculus we have (see \cite[formula (2.15)]{buosohinged})
\begin{equation*}
D^2(u\circ\phi^{-1})\circ\phi=(\nabla\phi)^{-t}D^2u(\nabla\phi)^{-1}+
	\left(\sum_{k,l=1}^N\frac{\partial u}{\partial x_k}\frac{\partial\sigma_{k,i}}{\partial x_l}\sigma_{l,j}\right)_{i,j},
\end{equation*}
where $\sigma=(\nabla\phi)^{-1}$. This yields the following formula
\begin{equation}\label{113}
d|_{\phi=\tilde{\phi}}(D^2(u\circ\phi^{-1})\circ\phi)[\psi]\circ\tilde{\phi}^{-1}=
	-D^2v\nabla\mu-\nabla\mu^tD^2v-\sum_{r=1}^N\frac{\partial v}{\partial y_r}D^2\mu_r,
\end{equation}
	where $\mu=\psi\circ\tilde{\phi}^{-1}$ and $v=u\circ\tilde{\phi}^{-1}$. We rewrite formula (\ref{113}) componentwise getting
\begin{multline*}
\left(d|_{\phi=\tilde{\phi}}(D^2(u\circ\phi^{-1})\circ\phi)[\psi]\circ\tilde{\phi}^{-1}\right)_{i,j}\\
=-\sum_{r=1}^N\left(\frac{\partial^2v}{\partial y_i\partial y_r}\frac{\partial\mu_r}{\partial y_j}
	+\frac{\partial^2v}{\partial y_j\partial y_r}\frac{\partial\mu_r}{\partial y_i}
	+\frac{\partial^2\mu_r}{\partial y_i\partial y_j}\frac{\partial v}{\partial y_r}\right).
\end{multline*}
	Moreover (see \cite[Lemma 3.26]{lala2004})
	\begin{equation*}
	\left(d|_{\phi=\tilde{\phi}}(\nabla(u\circ\phi^{-1})\circ\phi)[\psi]\circ\tilde{\phi}^{-1}\right)_{i}
=-\sum_{r=1}^N\frac{\partial v}{\partial y_r}\frac{\partial\mu_r}{\partial y_i}.
	\end{equation*}
	
	Now we use Einstein notation, dropping all the summation symbols. The first summand of the right hand side
	of (\ref{111}) equals
	\begin{equation}
	\label{115}
	-\int_{\tilde{\phi}(\Omega)}\left(\frac{\partial^2v_1}{\partial y_i\partial y_r}\frac{\partial\mu_r}{\partial y_j}
	+\frac{\partial^2v_1}{\partial y_j\partial y_r}\frac{\partial\mu_r}{\partial y_i}
	+\frac{\partial^2\mu_r}{\partial y_i\partial y_j}\frac{\partial v_1}{\partial y_r}\right)
	\frac{\partial^2v_2}{\partial y_i\partial y_j}dy.
	\end{equation}
	
In order to compute (\ref{115}), integrating by parts, we have
\begin{multline*}
\int_{\tilde{\phi}(\Omega)}\frac{\partial^2v_1}{\partial y_i\partial y_r}\frac{\partial\mu_r}{\partial y_j}\frac{\partial^2v_2}{\partial y_i\partial y_j}dy
=\int_{\partial\tilde{\phi}(\Omega)}\frac{\partial v_1}{\partial y_i}\frac{\partial \mu_r}{\partial y_j}\nu_r
	\frac{\partial^2v_2}{\partial y_i\partial y_j}d\sigma\\
-\int_{\tilde{\phi}(\Omega)}\frac{\partial v_1}{\partial y_i}\frac{\partial\mathrm{div}\mu}{\partial y_j}\frac{\partial^2v_2}{\partial y_i\partial y_j}dy
-\int_{\tilde{\phi}(\Omega)}\frac{\partial v_1}{\partial y_i}\frac{\partial \mu_r}{\partial y_j}\frac{\partial^3v_2}{\partial y_i\partial y_j\partial y_r}dy\\
=\int_{\partial\tilde{\phi}(\Omega)}\frac{\partial v_1}{\partial y_i}\frac{\partial \mu_r}{\partial y_j}\nu_r
	\frac{\partial^2v_2}{\partial y_i\partial y_j}d\sigma
	-\int_{\tilde{\phi}(\Omega)}\frac{\partial v_1}{\partial y_i}\frac{\partial \mu_r}{\partial y_j}\frac{\partial^3v_2}{\partial y_i\partial y_j\partial y_r}dy\\
-\int_{\partial\tilde{\phi}(\Omega)}\frac{\partial v_1}{\partial y_i}\mathrm{div}\mu\frac{\partial^2v_2}{\partial y_i\partial y_j}\nu_jd\sigma
+\int_{\tilde{\phi}(\Omega)} D^2v_1:D^2v_2\mathrm{div}\mu dy\\
+\int_{\tilde{\phi}(\Omega)}\mathrm{div}\mu\nabla v_1\cdot\nabla\Delta v_2 dy,
\end{multline*}
and 
\begin{multline*}
\int_{\tilde{\phi}(\Omega)}\frac{\partial v_1}{\partial y_r}\frac{\partial^2\mu_r}{\partial y_i\partial y_j}\frac{\partial^2v_2}{\partial y_i\partial y_j}dy
	=\int_{\partial\tilde{\phi}(\Omega)}\frac{\partial v_1}{\partial y_r}\frac{\partial\mu_r}{\partial y_i}\nu_j
		\frac{\partial^2v_2}{\partial y_i\partial y_j}d\sigma\\
-\int_{\tilde{\phi}(\Omega)}\frac{\partial^2v_1}{\partial y_r\partial y_j}\frac{\partial\mu_r}{\partial y_i}\frac{\partial^2v_2}{\partial y_i\partial y_j}dy
	-\int_{\tilde{\phi}(\Omega)}\frac{\partial v_1}{\partial y_r}\frac{\partial\mu_r}{\partial y_i}\frac{\partial\Delta v_2}{\partial y_i}dy\\
=\int_{\partial\tilde{\phi}(\Omega)}\frac{\partial v_1}{\partial y_r}\frac{\partial\mu_r}{\partial y_i}\nu_j
		\frac{\partial^2v_2}{\partial y_i\partial y_j}d\sigma -\int_{\tilde{\phi}(\Omega)}\frac{\partial v_1}{\partial y_r}\frac{\partial\mu_r}{\partial y_i}\frac{\partial\Delta v_2}{\partial y_i}dy\\
-\int_{\partial\tilde{\phi}(\Omega)}\frac{\partial v_1}{\partial y_j}\frac{\partial\mu_r}{\partial y_i}\nu_r\frac{\partial^2v_2}{\partial y_i\partial y_j}d\sigma
	+\int_{\tilde{\phi}(\Omega)}\frac{\partial v_1}{\partial y_j}\frac{\partial\mathrm{div}\mu}{\partial y_i}\frac{\partial^2v_2}{\partial y_i\partial y_j}dy\\
+\int_{\tilde{\phi}(\Omega)}\frac{\partial v_1}{\partial y_j}\frac{\partial\mu_r}{\partial y_i}\frac{\partial^3v_2}{\partial y_i\partial y_j\partial y_r}dy\\
=\int_{\partial\tilde{\phi}(\Omega)}\frac{\partial v_1}{\partial y_r}\frac{\partial\mu_r}{\partial y_i}\nu_j
		\frac{\partial^2v_2}{\partial y_i\partial y_j}d\sigma -\int_{\tilde{\phi}(\Omega)}\frac{\partial v_1}{\partial y_r}\frac{\partial\mu_r}{\partial y_i}\frac{\partial\Delta v_2}{\partial y_i}dy\\
-\int_{\partial\tilde{\phi}(\Omega)}\frac{\partial v_1}{\partial y_j}\frac{\partial\mu_r}{\partial y_i}\nu_r\frac{\partial^2v_2}{\partial y_i\partial y_j}d\sigma
+\int_{\tilde{\phi}(\Omega)}\frac{\partial v_1}{\partial y_j}\frac{\partial\mu_r}{\partial y_i}\frac{\partial^3v_2}{\partial y_i\partial y_j\partial y_r}dy\\
+\int_{\partial\tilde{\phi}(\Omega)}\frac{\partial v_1}{\partial y_j}\mathrm{div}\mu\frac{\partial^2v_2}{\partial y_i\partial y_j}\nu_id\sigma
	-\int_{\tilde{\phi}(\Omega)} D^2v_1:D^2v_2\mathrm{div}\mu dy\\
-\int_{\tilde{\phi}(\Omega)}\nabla v_1\cdot\nabla\Delta v_2\mathrm{div}\mu dy.
\end{multline*}
We also have
\begin{multline*}
\int_{\tilde{\phi}(\Omega)}\frac{\partial v_1}{\partial y_r}\frac{\partial \mu_r}{\partial y_i}\frac{\partial v_2}{\partial y_i}dy
=\int_{\partial\tilde{\phi}(\Omega)}\frac{\partial v_2}{\partial\nu}\nabla v_1\cdot\mu d\sigma
-\int_{\tilde{\phi}(\Omega)}\Delta v_2\nabla v_1\cdot\mu dy\\
-\int_{\tilde{\phi}(\Omega)}\frac{\partial v_2}{\partial y_i}\frac{\partial^2 v_1}{\partial y_i\partial y_r}\mu_rdy
=\int_{\partial\tilde{\phi}(\Omega)}\frac{\partial v_2}{\partial\nu}\nabla v_1\cdot\mu d\sigma
-\int_{\tilde{\phi}(\Omega)}\Delta v_2\nabla v_1\cdot\mu dy\\
-\int_{\partial\tilde{\phi}(\Omega)}\nabla v_1\cdot\nabla v_2 \mu\cdot\nu d\sigma
+\int_{\tilde{\phi}(\Omega)}\frac{\partial v_1}{\partial y_i}\frac{\partial^2 v_2}{\partial y_i\partial y_r}\mu_rdy
+\int_{\tilde{\phi}(\Omega)}\nabla v_1\cdot\nabla v_2\mathrm{div}\mu dy.
\end{multline*}
It follows that
\begin{multline}
\label{118}
d|_{\phi=\tilde{\phi}}\mathcal P^{\mathcal S}_{\phi}[\psi][u_1][u_2]\\
=-\int_{\tilde{\phi}(\Omega)} D^2v_1:D^2v_2\mathrm{div}\mu dy
-\int_{\partial\tilde{\phi}(\Omega)}\left(\frac{\partial v_1}{\partial y_i}	\frac{\partial^2v_2}{\partial y_i\partial y_j}+
		\frac{\partial v_2}{\partial y_i}	\frac{\partial^2v_1}{\partial y_i\partial y_j}\right)\frac{\partial \mu_r}{\partial y_j}\nu_rd\sigma\\
	+\int_{\tilde{\phi}(\Omega)}\left(\frac{\partial v_1}{\partial y_i}\frac{\partial^3v_2}{\partial y_i\partial y_j\partial y_r}+
		\frac{\partial v_2}{\partial y_i}\frac{\partial^3v_1}{\partial y_i\partial y_j\partial y_r}\right)\frac{\partial \mu_r}{\partial y_j}dy\\
+\int_{\partial\tilde{\phi}(\Omega)}\left(\frac{\partial v_1}{\partial y_i}\frac{\partial^2v_2}{\partial y_i\partial y_j}+
	\frac{\partial v_2}{\partial y_i}\frac{\partial^2v_1}{\partial y_i\partial y_j}\right)\nu_j\mathrm{div}\mu d\sigma\\
-\int_{\tilde{\phi}(\Omega)}\left(\nabla v_1\cdot\nabla\Delta v_2+\nabla v_2\cdot\nabla\Delta v_1\right)\mathrm{div}\mu dy\\
-\int_{\partial\tilde{\phi}(\Omega)}\left(\frac{\partial v_1}{\partial y_r}\frac{\partial^2v_2}{\partial y_i\partial y_j}+
	\frac{\partial v_2}{\partial y_r}\frac{\partial^2v_1}{\partial y_i\partial y_j}\right)\nu_j\frac{\partial\mu_r}{\partial y_i}d\sigma\\
		+\int_{\tilde{\phi}(\Omega)}\left(\frac{\partial v_1}{\partial y_r}\frac{\partial\Delta v_2}{\partial y_i}+
			\frac{\partial v_2}{\partial y_r}\frac{\partial\Delta v_1}{\partial y_i}\right)\frac{\partial\mu_r}{\partial y_i}dy\\
	-\tau\int_{\partial\tilde{\phi}(\Omega)}\left(\frac{\partial v_1}{\partial\nu}\nabla v_2+\frac{\partial v_2}{\partial\nu}\nabla v_1\right)\cdot\mu d\sigma\\
	+\tau\int_{\tilde{\phi}(\Omega)}\left(\Delta v_1\nabla v_2+\Delta v_2\nabla v_1\right)\cdot\mu dy
	+\tau\int_{\partial\tilde{\phi}(\Omega)}\nabla v_1\cdot\nabla v_2 \mu\cdot\nu d\sigma.
\end{multline}

Now we recall that
$$\mathrm{div}\mu=\mathrm{div}_{\partial\tilde{\phi}(\Omega)}\mu+\frac{\partial\mu}{\partial\nu}\cdot\nu\ \ {\rm on}\  \partial\tilde\phi(\Omega),$$
(see also \cite[\S 8.5]{delfour}) and that, since $\nu=\nabla b$, where $b$ is the distance from the boundary defined in an appropriate tubular neighborhood of the boundary, then
$\nabla\nu=(\nabla\nu)^t$ and $\frac{\partial\nu}{\partial\nu}=0$, from which it follows that
$$\nabla_{\partial\tilde{\phi}(\Omega)}\nu=(\nabla_{\partial\tilde{\phi}(\Omega)}\nu)^t\ \ {\rm on}\ \partial\tilde\phi(\Omega).$$
We will use these identities throughout all the following computations.

Using the fact that
$$\frac{\partial^2v_1}{\partial\nu^2}=\frac{\partial^2v_2}{\partial\nu^2}=0\ \ {\rm on}\ \partial\tilde\phi(\Omega),$$
we get that the sixth summand in (\ref{118}) equals
\begin{multline}
-\int_{\partial\tilde{\phi}(\Omega)}\left(\frac{\partial v_1}{\partial y_r}(D^2v_2.\nu)_{\partial\tilde{\phi}(\Omega)}+\frac{\partial v_2}{\partial y_r}(D^2v_1.\nu)_{\partial\tilde{\phi}(\Omega)}
	\right)\cdot\nabla_{\partial\tilde{\phi}(\Omega)}\mu_rd\sigma\\
-\int_{\partial\tilde{\phi}(\Omega)}\left(\frac{\partial v_1}{\partial y_r}\frac{\partial^2 v_2}{\partial\nu^2}+\frac{\partial v_2}{\partial y_r}\frac{\partial^2 v_1}{\partial\nu^2}
	\right)\frac{\partial\mu_r}{\partial\nu}d\sigma\\
=	\int_{\partial\tilde{\phi}(\Omega)}\left(\nabla_{\partial\tilde{\phi}(\Omega)}\left(\frac{\partial v_1}{\partial y_r}\right)(D^2v_2.\nu)_{\partial\tilde{\phi}(\Omega)}+\nabla_{\partial\tilde{\phi}(\Omega)}\left(\frac{\partial v_2}{\partial y_r}\right)(D^2v_1.\nu)_{\partial\tilde{\phi}(\Omega)}
	\right)\mu_rd\sigma\\
+\int_{\partial\tilde{\phi}(\Omega)}\left(\mathrm{div}_{\partial\tilde{\phi}(\Omega)}(D^2v_1.\nu)_{\partial\tilde{\phi}(\Omega)}\nabla v_2+\mathrm{div}_{\partial\tilde{\phi}(\Omega)}(D^2v_2.\nu)_{\partial\tilde{\phi}(\Omega)}\nabla v_1\right)\cdot\mu d\sigma\\	
=\int_{\partial\tilde{\phi}(\Omega)}\left(\frac{\partial^2 v_1}{\partial y_i\partial y_r}\frac{\partial^2 v_2}{\partial y_i\partial y_j}
	+\frac{\partial^2 v_2}{\partial y_i\partial y_r}\frac{\partial^2 v_1}{\partial y_i\partial y_j}\right)
	\nu_j\mu_rd\sigma\\
+\int_{\partial\tilde{\phi}(\Omega)}\left(\mathrm{div}_{\partial\tilde{\phi}(\Omega)}(D^2v_1.\nu)_{\partial\tilde{\phi}(\Omega)}\nabla v_2+\mathrm{div}_{\partial\tilde{\phi}(\Omega)}(D^2v_2.\nu)_{\partial\tilde{\phi}(\Omega)}\nabla v_1\right)\cdot\mu d\sigma.
\end{multline}
The seventh summand in (\ref{118}) equals
\begin{multline}
\int_{\partial\tilde{\phi}(\Omega)}\left(\frac{\partial\Delta v_1}{\partial\nu}\nabla v_2+\frac{\partial\Delta v_2}{\partial\nu}\nabla v_1\right)\cdot\mu d\sigma
-\int_{\tilde{\phi}(\Omega)}\left(\Delta^2v_1\nabla v_2+\Delta^2v_2\nabla v_1\right)\cdot\mu d\sigma\\
-\int_{\tilde{\phi}(\Omega)}\left(\frac{\partial^2v_1}{\partial y_i \partial y_r}\frac{\partial\Delta v_2}{\partial y_i}+\frac{\partial^2v_2}{\partial y_i \partial y_r}\frac{\partial\Delta v_1}{\partial y_i}\right)\mu_r dy\\
=\int_{\partial\tilde{\phi}(\Omega)}\left(\frac{\partial\Delta v_1}{\partial\nu}\nabla v_2+\frac{\partial\Delta v_2}{\partial\nu}\nabla v_1\right)\cdot\mu d\sigma
-\int_{\tilde{\phi}(\Omega)}\left(\Delta^2v_1\nabla v_2+\Delta^2v_2\nabla v_1\right)\cdot\mu d\sigma\\
-\int_{\partial\tilde{\phi}(\Omega)}\left(\nabla v_1\cdot\nabla\Delta v_2+\nabla v_2\cdot\nabla\Delta v_1\right)\mu\cdot\nu d\sigma
+\int_{\tilde{\phi}(\Omega)}\left(\frac{\partial v_1}{\partial y_i}\frac{\partial^2\Delta v_2}{\partial y_i\partial y_r}+\frac{\partial v_2}{\partial y_i}\frac{\partial^2\Delta v_1}{\partial y_i\partial y_r}\right)\mu_r dy\\
+\int_{\tilde{\phi}(\Omega)}\left(\nabla v_1\cdot\nabla\Delta v_2+\nabla v_2\cdot\nabla\Delta v_1\right)\mathrm{div}\mu dy.
\end{multline}
The second summand in (\ref{118}) equals
\begin{multline}
-\int_{\partial\tilde{\phi}(\Omega)}\nabla(\nabla v_1\cdot\nabla v_2)\nabla(\mu_r)\nu_rd\sigma\\
=-\int_{\partial\tilde{\phi}(\Omega)}\nabla_{\partial\tilde{\phi}(\Omega)}(\nabla v_1\cdot\nabla v_2)\nabla_{\partial\tilde{\phi}(\Omega)}(\mu_r)\nu_rd\sigma
-\int_{\partial\tilde{\phi}(\Omega)}\frac{\partial\ }{\partial\nu}(\nabla v_1\cdot\nabla v_2)\frac{\partial\mu_r}{\partial\nu}\nu_rd\sigma.
\end{multline}
The third summand in (\ref{118}) equals
\begin{multline}
\label{122}
\int_{\partial\tilde{\phi}(\Omega)}\left(\frac{\partial v_1}{\partial y_i}\frac{\partial^3v_2}{\partial y_i\partial y_j\partial y_r}+
		\frac{\partial v_2}{\partial y_i}\frac{\partial^3v_1}{\partial y_i\partial y_j\partial y_r}\right)\nu_j \mu_rd\sigma\\
-\int_{\tilde{\phi}(\Omega)}\left(\frac{\partial v_1}{\partial y_i}\frac{\partial^2\Delta v_2}{\partial y_i\partial y_r}+\frac{\partial v_2}{\partial y_i}\frac{\partial^2\Delta v_1}{\partial y_i\partial y_r}\right)\mu_r dy\\
-\int_{\tilde{\phi}(\Omega)}\left(\frac{\partial^2 v_1}{\partial y_i\partial y_j}\frac{\partial^3v_2}{\partial y_i\partial y_j\partial y_r}+
		\frac{\partial^2 v_2}{\partial y_i\partial y_j}\frac{\partial^3v_1}{\partial y_i\partial y_j\partial y_r}\right)\mu_rdy\\
=\int_{\partial\tilde{\phi}(\Omega)}\left(\frac{\partial v_1}{\partial y_i}\frac{\partial^3v_2}{\partial y_i\partial y_j\partial y_r}+
		\frac{\partial v_2}{\partial y_i}\frac{\partial^3v_1}{\partial y_i\partial y_j\partial y_r}\right)\nu_j \mu_rd\sigma\\
-\int_{\tilde{\phi}(\Omega)}\left(\frac{\partial v_1}{\partial y_i}\frac{\partial^2\Delta v_2}{\partial y_i\partial y_r}+\frac{\partial v_2}{\partial y_i}\frac{\partial^2\Delta v_1}{\partial y_i\partial y_r}\right)\mu_r dy\\
-\int_{\partial\tilde{\phi}(\Omega)}D^2v_1:D^2v_2\mu\cdot\nu d\sigma
+\int_{\tilde{\phi}(\Omega)}D^2v_1:D^2v_2\mathrm{div}\mu dy.
\end{multline}
From (\ref{118})-(\ref{122}), it follows that
\begin{multline}
\label{123}
d|_{\phi=\tilde{\phi}}\mathcal P^{\mathcal S}_{\phi}[\psi][u_1][u_2]
=-\int_{\partial\tilde{\phi}(\Omega)}\nabla_{\partial\tilde{\phi}(\Omega)}(\nabla v_1\cdot\nabla v_2)\nabla_{\partial\tilde{\phi}(\Omega)}(\mu_r)\nu_rd\sigma\\
-\int_{\partial\tilde{\phi}(\Omega)}\frac{\partial\ }{\partial\nu}(\nabla v_1\cdot\nabla v_2)\frac{\partial\mu_r}{\partial\nu}\nu_rd\sigma
+\int_{\partial\tilde{\phi}(\Omega)}\frac{\partial\ }{\partial\nu}(\nabla v_1\cdot\nabla v_2)\mathrm{div}\mu d\sigma\\
+\int_{\partial\tilde{\phi}(\Omega)}\left(\mathrm{div}_{\partial\tilde{\phi}(\Omega)}(D^2v_1.\nu)_{\partial\tilde{\phi}(\Omega)}\nabla v_2+\mathrm{div}_{\partial\tilde{\phi}(\Omega)}(D^2v_2.\nu)_{\partial\tilde{\phi}(\Omega)}\nabla v_1\right)\cdot\mu d\sigma\\
+\int_{\partial\tilde{\phi}(\Omega)}\left(\frac{\partial^2 v_1}{\partial y_i\partial y_r}\frac{\partial^2 v_2}{\partial y_i\partial y_j}
	+\frac{\partial^2 v_2}{\partial y_i\partial y_r}\frac{\partial^2 v_1}{\partial y_i\partial y_j}\right)
	\nu_j\mu_rd\sigma\\
+\int_{\partial\tilde{\phi}(\Omega)}\left(\frac{\partial v_1}{\partial y_i}\frac{\partial^3v_2}{\partial y_i\partial y_j\partial y_r}+
		\frac{\partial v_2}{\partial y_i}\frac{\partial^3v_1}{\partial y_i\partial y_j\partial y_r}\right)\nu_j \mu_rd\sigma\\
-\int_{\partial\tilde{\phi}(\Omega)}D^2v_1:D^2v_2\mu\cdot\nu d\sigma
-\int_{\partial\tilde{\phi}(\Omega)}\left(\nabla v_1\cdot\nabla\Delta v_2+\nabla v_2\cdot\nabla\Delta v_1\right)\mu\cdot\nu d\sigma\\
+\int_{\partial\tilde{\phi}(\Omega)}\left(\frac{\partial\Delta v_1}{\partial\nu}\nabla v_2+\frac{\partial\Delta v_2}{\partial\nu}\nabla v_1\right)\cdot\mu d\sigma
-\int_{\tilde{\phi}(\Omega)}\left(\Delta^2v_1\nabla v_2+\Delta^2v_2\nabla v_1\right)\cdot\mu d\sigma\\
	-\tau\int_{\partial\tilde{\phi}(\Omega)}\left(\frac{\partial v_1}{\partial\nu}\nabla v_2+\frac{\partial v_2}{\partial\nu}\nabla v_1\right)\cdot\mu d\sigma\\
	+\tau\int_{\tilde{\phi}(\Omega)}\left(\Delta v_1\nabla v_2+\Delta v_2\nabla v_1\right)\cdot\mu dy
	+\tau\int_{\partial\tilde{\phi}(\Omega)}\nabla v_1\cdot\nabla v_2 \mu\cdot\nu d\sigma\\
=-\int_{\partial\tilde{\phi}(\Omega)}\nabla_{\partial\tilde{\phi}(\Omega)}(\nabla v_1\cdot\nabla v_2)\nabla_{\partial\tilde{\phi}(\Omega)}(\mu_r)\nu_rd\sigma
+\int_{\partial\tilde{\phi}(\Omega)}\frac{\partial\ }{\partial\nu}(\nabla v_1\cdot\nabla v_2)\mathrm{div}_{\partial\tilde{\phi}(\Omega)}\mu d\sigma\\
+\int_{\partial\tilde{\phi}(\Omega)}\left(\mathrm{div}_{\partial\tilde{\phi}(\Omega)}(D^2v_1.\nu)_{\partial\tilde{\phi}(\Omega)}\nabla v_2+\mathrm{div}_{\partial\tilde{\phi}(\Omega)}(D^2v_2.\nu)_{\partial\tilde{\phi}(\Omega)}\nabla v_1\right)\cdot\mu d\sigma\\
+\int_{\partial\tilde{\phi}(\Omega)}\left(\frac{\partial\Delta v_1}{\partial\nu}\nabla v_2+\frac{\partial\Delta v_2}{\partial\nu}\nabla v_1\right)\cdot\mu d\sigma
-\tau\int_{\partial\tilde{\phi}(\Omega)}\left(\frac{\partial v_1}{\partial\nu}\nabla v_2+\frac{\partial v_2}{\partial\nu}\nabla v_1\right)\cdot\mu d\sigma\\
+\int_{\partial\tilde{\phi}(\Omega)}\frac{\partial\ }{\partial\nu}\left(\frac{\partial\ }{\partial y_r}(\nabla v_1\cdot\nabla v_2)\right)\mu_r d\sigma\\
-\int_{\partial\tilde{\phi}(\Omega)}D^2v_1:D^2v_2\mu\cdot\nu d\sigma
-\int_{\partial\tilde{\phi}(\Omega)}(\nabla v_1\cdot\nabla\Delta v_2+\nabla v_2\cdot\nabla\Delta v_1)\mu\cdot\nu d\sigma\\
-\int_{\tilde{\phi}(\Omega)}\left(\Delta^2v_1\nabla v_2+\Delta^2v_2\nabla v_1\right)\cdot\mu d\sigma\\
	+\tau\int_{\tilde{\phi}(\Omega)}\left(\Delta v_1\nabla v_2+\Delta v_2\nabla v_1\right)\cdot\mu dy
	+\tau\int_{\partial\tilde{\phi}(\Omega)}\nabla v_1\cdot\nabla v_2 \mu\cdot\nu d\sigma.
\end{multline}

The first summand on the right hand side of (\ref{123}) equals
\begin{equation*}
\int_{\partial\tilde{\phi}(\Omega)}\Delta_{\partial\tilde{\phi}(\Omega)}(\nabla v_1\cdot\nabla v_2)\mu\cdot\nu d\sigma
+\int_{\partial\tilde{\phi}(\Omega)}\nabla_{\partial\tilde{\phi}(\Omega)}(\nabla v_1\cdot\nabla v_2)\cdot(\nabla_{\partial\tilde{\phi}(\Omega)}
	\nu_r)\mu_rd\sigma,
\end{equation*}
while the sixth one equals
\begin{multline*}
\int_{\partial\tilde{\phi}(\Omega)}\frac{\partial^2\ }{\partial\nu^2}(\nabla v_1\cdot\nabla v_2)\mu\cdot\nu d\sigma
+\int_{\partial\tilde{\phi}(\Omega)}\nabla_{\partial\tilde{\phi}(\Omega)}\left(\frac{\partial\ }{\partial\nu}(\nabla v_1\cdot\nabla v_2)\right)\cdot\mu d\sigma\\
-\int_{\partial\tilde{\phi}(\Omega)}\nabla_{\partial\tilde{\phi}(\Omega)}(\nabla v_1\cdot\nabla v_2)\cdot(\nabla_{\partial\tilde{\phi}(\Omega)}
	\nu_r)\mu_rd\sigma.
\end{multline*}
Using the fact that
$$\int_{\partial\tilde{\phi}(\Omega)}\mathrm{div}_{\partial\tilde{\phi}(\Omega)}\left(\frac{\partial\ }{\partial\nu}(\nabla v_1\cdot\nabla v_2)\cdot\mu\right)d\sigma
=\int_{\partial\tilde{\phi}(\Omega)}K\frac{\partial\ }{\partial\nu}(\nabla v_1\cdot\nabla v_2)\mu\cdot\nu d\sigma,$$
where $K$ denotes the mean curvature of $\partial\tilde{\phi}(\Omega)$ (see \cite[\S 8.5]{delfour}), we obtain
\begin{multline*}
d|_{\phi=\tilde{\phi}}\mathcal P^{\mathcal S}_{\phi}[\psi][u_1][u_2]
=\int_{\partial\tilde{\phi}(\Omega)}\Delta_{\partial\tilde{\phi}(\Omega)}(\nabla v_1\cdot\nabla v_2)\mu\cdot\nu d\sigma\\
+\int_{\partial\tilde{\phi}(\Omega)}K\frac{\partial\ }{\partial\nu}(\nabla v_1\cdot\nabla v_2)\mu\cdot\nu d\sigma
+\int_{\partial\tilde{\phi}(\Omega)}\frac{\partial^2\ }{\partial\nu^2}(\nabla v_1\cdot\nabla v_2)\mu\cdot\nu d\sigma\\
-\int_{\partial\tilde{\phi}(\Omega)}D^2v_1:D^2v_2\mu\cdot\nu d\sigma
-\int_{\partial\tilde{\phi}(\Omega)}(\nabla v_1\cdot\nabla\Delta v_2+\nabla v_2\cdot\nabla\Delta v_1)\mu\cdot\nu d\sigma\\
+\int_{\partial\tilde{\phi}(\Omega)}\left(\mathrm{div}_{\partial\tilde{\phi}(\Omega)}(D^2v_1.\nu)_{\partial\tilde{\phi}(\Omega)}\nabla v_2+\mathrm{div}_{\partial\tilde{\phi}(\Omega)}(D^2v_2.\nu)_{\partial\tilde{\phi}(\Omega)}\nabla v_1\right)\cdot\mu d\sigma\\
+\int_{\partial\tilde{\phi}(\Omega)}\left(\frac{\partial\Delta v_1}{\partial\nu}\nabla v_2+\frac{\partial\Delta v_2}{\partial\nu}\nabla v_1\right)\cdot\mu d\sigma
-\tau\int_{\partial\tilde{\phi}(\Omega)}\left(\frac{\partial v_1}{\partial\nu}\nabla v_2+\frac{\partial v_2}{\partial\nu}\nabla v_1\right)\cdot\mu d\sigma\\
-\int_{\tilde{\phi}(\Omega)}\left(\Delta^2v_1\nabla v_2+\Delta^2v_2\nabla v_1\right)\cdot\mu d\sigma\\
	+\tau\int_{\tilde{\phi}(\Omega)}\left(\Delta v_1\nabla v_2+\Delta v_2\nabla v_1\right)\cdot\mu dy
	+\tau\int_{\partial\tilde{\phi}(\Omega)}\nabla v_1\cdot\nabla v_2 \mu\cdot\nu d\sigma\\
=\int_{\partial\tilde{\phi}(\Omega)}\Delta(\nabla v_1\cdot\nabla v_2)\mu\cdot\nu d\sigma
-\int_{\partial\tilde{\phi}(\Omega)}D^2v_1:D^2v_2\mu\cdot\nu d\sigma\\
-\int_{\partial\tilde{\phi}(\Omega)}(\nabla v_1\cdot\nabla\Delta v_2+\nabla v_2\cdot\nabla\Delta v_1)\mu\cdot\nu d\sigma\\
+\int_{\partial\tilde{\phi}(\Omega)}\left(\mathrm{div}_{\partial\tilde{\phi}(\Omega)}(D^2v_1.\nu)_{\partial\tilde{\phi}(\Omega)}\nabla v_2+\mathrm{div}_{\partial\tilde{\phi}(\Omega)}(D^2v_2.\nu)_{\partial\tilde{\phi}(\Omega)}\nabla v_1\right)\cdot\mu d\sigma\\
+\int_{\partial\tilde{\phi}(\Omega)}\left(\frac{\partial\Delta v_1}{\partial\nu}\nabla v_2+\frac{\partial\Delta v_2}{\partial\nu}\nabla v_1\right)\cdot\mu d\sigma
-\tau\int_{\partial\tilde{\phi}(\Omega)}\left(\frac{\partial v_1}{\partial\nu}\nabla v_2+\frac{\partial v_2}{\partial\nu}\nabla v_1\right)\cdot\mu d\sigma\\
-\int_{\tilde{\phi}(\Omega)}\left((\Delta^2v_1-\tau\Delta v_1)\nabla v_2+(\Delta^2v_2-\tau\Delta v_2)\nabla v_1\right)\cdot\mu d\sigma\\
	+\tau\int_{\partial\tilde{\phi}(\Omega)}\nabla v_1\cdot\nabla v_2 \mu\cdot\nu d\sigma.
\end{multline*}
Using the equality
$$
\Delta(\nabla v_1\cdot\nabla v_2)=\nabla\Delta v_1\cdot\nabla v_2+\nabla v_1\cdot\nabla\Delta v_2+2D^2v_1:D^2v_2
$$
we finally get formula (\ref{formulafondamentale}).
\endproof
\end{lem}

Now we can compute Hadamard-type formulas for the eigenvalues of problem (\ref{steklovphi}).
\begin{thm}
	\label{duesette}
		Let $\Omega$ be a bounded domain in $\mathbb{R}^N$ of class $C^1$. Let
		$F$ be a finite non-empty subset of $\mathbb{N}$. Let $\tilde{\phi}\in\Theta_{\Omega}[F]$ be such that $\partial\tilde{\phi}(\Omega)
		\in C^4$.
		Let $v_1,\dots,v_{|F|}$ be an orthonormal basis of the eigenspace associated with the eigenvalue $\lambda_F[\tilde{\phi}]$ of
		problem (\ref{steklovphi}) in $L^2(\partial\tilde{\phi}(\Omega))$. Then
		\begin{multline*}
			d|_{\phi=\tilde{\phi}}(\Lambda_{F,s})[\psi]=-\lambda_F^{s-1}[\tilde{\phi}]\binom{|F|-1}{s-1}
\sum_{l=1}^{|F|}\int_{\partial\tilde{\phi}(\Omega)}\Big(\lambda_FKv_l^2\\
+\lambda_F\frac{\partial(v_l^2)}{\partial\nu}-\tau|\nabla v_l|^2-|D^2v_l|^2\Big)\mu\cdot\nu d\sigma,
		\end{multline*}
		for all $\psi\in(C^2(\overline{\Omega}))^N$, where $\mu=\psi\circ\tilde{\phi}^{(-1)}$, and $K$ denotes the mean curvature of $\partial\tilde{\phi}(\Omega)$.
	
	\proof
	First of all we note that $v_1,...,v_{|F|}\in H^4(\tilde\phi(\Omega))$ (see e.g., \cite[\S 2.5]{ggs}). We set $u_l=v_l\circ\tilde{\phi}$ for $l=1,\dots,|F|$. For $|F|>1$ (case $|F|=1$ is similar), $s\leq|F|$, we have
	\begin{equation}
	\label{formula}
	d|_{\phi=\tilde{\phi}}(\Lambda_{F,s})[\psi]=
	-\lambda_F^{s+2}[\tilde{\phi}]\binom{|F|-1}{s-1}\sum_{l=1}^{|F|}\mathcal P^{\mathcal S}_{\tilde{\phi}}\left[
		d|_{\phi=\tilde{\phi}}W^{\mathcal S}_{\phi}[\psi][p(u_l)]\right]\left[p(u_l)\right].
	\end{equation}
	We refer to \cite[Theorem 3.38]{lala2004} for a proof of formula (\ref{formula}).
	
	By standard calculus in normed spaces we have:
	\begin{multline*}
			\mathcal P^{\mathcal S}_{\tilde{\phi}}\left[\mathrm{d}|_{\phi=\tilde{\phi}}\left((\pi_{\phi}^{\sharp,\mathcal S})^{-1}\circ\left(
				\mathcal P^{\mathcal S}_{\phi}\right)^{-1}\circ{\mathcal J}^{\mathcal S}_{\phi}\circ\mathrm{Tr}\circ\pi_{\phi}^{\sharp,\mathcal S}\right)[\psi][p(u_l)]\right]\left[p(u_l)\right]  \\ 
			= \mathcal P^{\mathcal S}_{\tilde{\phi}}\left[(\pi_{\tilde{\phi}}^{\sharp,\mathcal S})^{-1}\circ\left(\mathcal P^{\mathcal S}_{\tilde{\phi}}\right)^{-1}\circ
							\mathrm{d}|_{\phi=\tilde{\phi}}\left({\mathcal J}^{\mathcal S}_{\phi}\circ\mathrm{Tr}\circ\pi_{\phi}^{\sharp,\mathcal S}\right)[\psi][p(u_l)]\right]\left[p(u_l)\right]\\
			 						+\mathcal P^{\mathcal S}_{\tilde{\phi}}\left[\mathrm{d}|_{\phi=\tilde{\phi}}\left((\pi_{\phi}^{\sharp,\mathcal S})^{-1}\circ\left(
				\mathcal P^{\mathcal S}_{\phi}\right)^{-1}\right)[\psi]\circ{\mathcal J}^{\mathcal S}_{\tilde{\phi}}\circ\mathrm{Tr}\circ\pi_{\tilde{\phi}}^{\sharp,\mathcal S}[p(u_l)]\right]\left[p(u_l)\right].
	\end{multline*}
	
	Now note that:
	\begin{multline*}
		\mathcal P^{\mathcal S}_{\tilde{\phi}}\left[	(\pi_{\tilde{\phi}}^{\sharp,\mathcal S})^{-1}\circ\left(\mathcal P^{\mathcal S}_{\tilde{\phi}}\right)^{-1}\circ
							\mathrm{d}|_{\phi=\tilde{\phi}}\left({\mathcal J}^{\mathcal S}_{\phi}\circ\mathrm{Tr}\circ\pi_{\phi}^{\sharp,\mathcal S}\right)[\psi][p(u_l)]\right]\left[p(u_l)\right]\\
			=\int_{\partial\tilde{\phi}(\Omega)}\left(Kv_l^2+\frac{\partial(v_l^2)}{\partial\nu}\right)\mu\cdot\nu d\sigma
			-\int_{\partial\tilde{\phi}(\Omega)}\nabla(v_l^2)\cdot\mu d\sigma,
	\end{multline*}
	(see also \cite[Lemma 3.3]{lambertisteklov}) and
	\begin{multline*}
		\mathcal P^{\mathcal S}_{\tilde{\phi}}\left[\mathrm{d}|_{\phi=\tilde{\phi}}\left((\pi_{\phi}^{\sharp,\mathcal S})^{-1}\circ\left(
				\mathcal P^{\mathcal S}_{\phi}\right)^{-1}\right)[\psi]\circ{\mathcal J}^{\mathcal S}_{\tilde{\phi}}\circ\mathrm{Tr}\circ\pi_{\tilde{\phi}}^{\sharp,\mathcal S}[p(u_l)]\right]\left[p(u_l)\right]\\
				=-\lambda_F^{-1}d|_{\phi=\tilde{\phi}}\left(\mathcal P^{\mathcal S}_{\phi}\circ\pi^{\mathcal S}_{\phi}\right)[\psi][u_l][\pi^{\mathcal S}_{\tilde{\phi}}(u_l)].
	\end{multline*}
	(We refer to \cite[Lemma 2.4]{lala2007} for more explicit computations). Using formula (\ref{formulafondamentale}) we obtain
	\begin{multline*}
		\mathcal P^{\mathcal S}_{\tilde{\phi}}\left[\mathrm{d}|_{\phi=\tilde{\phi}}\left((\pi_{\phi}^{\sharp,\mathcal S})^{-1}\circ\left(
				\mathcal P^{\mathcal S}_{\phi}\right)^{-1}\right)[\psi]\circ{\mathcal J}^{\mathcal S}_{\tilde{\phi}}\circ\mathrm{Tr}\circ\pi_{\tilde{\phi}}^{\sharp,\mathcal S}[p(u_l)]\right]\left[p(u_l)\right]\\
		=-\lambda_F^{-1}\int_{\partial\tilde{\phi}(\Omega)}\left(|D^2v_l|^2+\tau|\nabla v_l|^2\right)\mu\cdot\nu d\sigma
+\int_{\partial\tilde{\phi}(\Omega)}\nabla(v_l^2)\cdot\mu d\sigma.
\end{multline*}
This concludes the proof.
	\endproof
\end{thm}

Now we turn our attention to extremum problems of the type
$$
\min_{\mathcal{V}(\phi)=\mathrm{const.}}\Lambda_{F,s}[\phi]{\rm \ or\ }\max_{\mathcal{V}(\phi)=\mathrm{const.}}\Lambda_{F,s}[\phi],
$$
where $\mathcal{V}(\phi)$ denotes the measure of $\phi(\Omega)$, i.e.,
	\begin{equation}
	\label{vfi}
	\mathcal{V}(\phi):=\int_{\phi(\Omega)}dx=\int_{\Omega}|\det D\phi|dx.
	\end{equation}
In particular, all $\phi$'s realizing the extremum are critical points under measure constraint, i.e., ${\rm Ker}\,d\mathcal V(\phi)\subseteq{\rm Ker}\,d\Lambda_{F,s}[\phi]$.
	We have the following result (see \cite[Proposition 2.10]{lalacri}).
	
	\begin{prop}
	\label{vfiprop}
	Let $\Omega$ be a bounded domain in $\mathbb{R}^N$ of class $C^1$. Then the following statements hold.
	\begin{enumerate}[i)]
		\item The map $\mathcal{V}$ from $\Phi(\Omega)$ to $\mathbb{R}$ defined in (\ref{vfi}) is real analytic. Moreover,
			the differential of $\mathcal{V}$ at $\tilde{\phi}\in\Phi(\Omega)$ is given by the formula
			\begin{equation*}
			d|_{\phi=\tilde{\phi}}\mathcal{V}(\phi)[\psi]=\int_{\tilde{\phi}(\Omega)}\mathrm{div}(\psi\circ\tilde{\phi}^{-1})dy=
			\int_{\partial\tilde{\phi}(\Omega)}(\psi\circ\tilde{\phi}^{-1})\cdot\nu d\sigma.
			\end{equation*}
		\item For $\mathcal{V}_0\in]0,+\infty[$, let
			$$V(\mathcal{V}_0):=\{\phi\in\Phi(\Omega):\mathcal{V}(\phi)=\mathcal{V}_0\}.$$
			If $V(\mathcal{V}_0)\neq\emptyset$, then $V(\mathcal{V}_0)$ is a real analytic manifold of $(C^2(\overline{\Omega}))^N$ of codimension $1$.
	\end{enumerate}
	\end{prop}
	
	Using Lagrange Multipliers Theorem, it is easy to prove the following
	
	\begin{thm}
	\label{moltiplicatori}
	Let $\Omega$ be a bounded domain in $\mathbb{R}^N$ of class $C^1$. Let $F$ be a non-empty finite
	subset of $\mathbb{N}$. Let $\mathcal{V}_0\in]0,+\infty[$. Let $\tilde{\phi}\in V(\mathcal{V}_0)$ be such that
	$\partial\tilde{\phi}(\Omega)\in C^4$ and $\lambda_j[\tilde{\phi}]$ have a common value $\lambda_F[\tilde{\phi}]$ for all $j\in F$ and
	$\lambda_l[\tilde{\phi}]\neq\lambda_F[\tilde{\phi}]$
	for all $l\in\mathbb{N}\setminus F$. For $s=1,\dots, |F|$, the function $\tilde{\phi}$ is a critical point for $\Lambda_{F,s}$ on
	$V(\mathcal{V}_0)$ if and only if there exists an orthonormal basis $v_1,\dots,v_{|F|}$ of the eigenspace corresponding to the eigenvalue
	$\lambda_F[\tilde{\phi}]$ of problem (\ref{steklovphi}) in $L^2(\partial\tilde{\phi}(\Omega))$, and a constant $c\in\mathbb{R}$ such that
	\begin{equation}
	\label{lacondizione}
	\sum_{l=1}^{|F|}\left(\lambda_F[\tilde{\phi}]\left(Kv_l^2+\frac{\partial \left(v_l^2\right)}{\partial\nu}\right)-\tau|\nabla v_l|^2-|D^2v_l|^2\right)=c,
	\mathrm{\ a.e.\ on\ }\partial\tilde{\phi}(\Omega).
	\end{equation}
	\end{thm}

	Now that we have a characterization for the criticality of $\tilde\phi$, we may wonder whether balls are critical domains. This is the aim of the following

	\begin{thm}
	Let $\Omega$ be a bounded domain of $\mathbb{R}^N$ of class $C^1$. Let $\tilde{\phi}\in\Phi(\Omega)$ be such that
	$\tilde{\phi}(\Omega)$ is a ball. Let $\tilde{\lambda}$ be an eigenvalue of problem (\ref{steklovphi}) in $\tilde{\phi}(\Omega)$,
	and let $F$ be the set of $j\in\mathbb{N}$ such that $\lambda_j[\tilde{\phi}]=\tilde{\lambda}$.
	Then $\Lambda_{F,s}$ has a critical point at $\tilde{\phi}$ on $V(\mathcal{V}(\tilde{\phi}))$,
	for all $s=1,\dots,|F|$.
	\proof
Using Lemma \ref{propedeutico} below and the fact that the mean curvature is constant for a ball,
	condition (\ref{lacondizione}) is immediately seen to be satisfied.
	\endproof
	\end{thm}

	\begin{lem}
	\label{propedeutico}
		Let $B$ be the unit ball in $\mathbb{R}^N$ centered at zero, and let $\lambda$ be an eigenvalue of problem (\ref{steklovphi}) in $B$. Let $F$ be the subset of
		$\mathbb{N}$ of all indeces $j$ such that the $j$-th eigenvalue of problem (\ref{steklovphi}) in $B$ coincides with $\lambda$. Let $v_1,\dots,v_{|F|}$
		be an orthonormal basis of the eigenspace associated with the eigenvalue $\lambda$, where the orthonormality is taken with respect to the scalar
		product in $L^2(\partial B)$. Then
		$$
		\sum_{j=1}^{|F|}v_j^2,\ \sum_{j=1}^{|F|}|\nabla v_j|^2,\ \sum_{j=1}^{|F|}|D^2v_j|^2
		$$ 
	are radial functions.
		
	\proof
		Let $O_N(\mathbb{R})$ denote the group of orthogonal linear transformations in $\mathbb{R}^N$. Since the Laplace operator is invariant under rotations, then
		$v_k\circ A$, where $A\in O_N(\mathbb R)$, is still an eigenfunction with eigenvalue $\lambda$; moreover, $\{v_j\circ A:j=1,\dots, |F|\}$ is another orthonormal basis
		for the eigenspace associated with to $\lambda$. Since both $\{v_j:j=1,\dots, |F|\}$ and $\{v_j\circ A:j=1,\dots, |F|\}$
		are orthonormal bases, then there exists $R[A]\in O_N(\mathbb{R})$ with matrix $(R_{ij}[A])_{i,j=1,\dots,|F|}$ such that
		\begin{equation*}
		v_j=\sum_{l=1}^{|F|}R_{jl}[A]v_l\circ A.
		\end{equation*}
		This implies that
		\begin{equation*}
		\sum_{j=1}^{|F|}v_j^2=\sum_{j=1}^{|F|}(v_j\circ A)^2,
		\end{equation*}
		from which we get that $\sum_{j=1}^{|F|}v_j^2$ is radial.
		Moreover, using standard calculus, we get
		\begin{multline*}
		\sum_{j=1}^{|F|}|\nabla v_j|^2=\sum_{l_1,l_2=1}^{|F|}R_{jl_1}[A]R_{jl_2}[A]\left(\nabla v_{l_1}\circ A\right)\cdot\left(\nabla v_{l_2}\circ A\right)
		=\sum_{l=1}^{|F|}|\nabla v_l\circ A|^2,
		\end{multline*}
		and		
		\begin{multline*}
		D^2v_j\cdot D^2v_j=\sum_{l_1,l_2=1}^{|F|}R_{jl_1}[A]R_{jl_2}[A]A^t\cdot (D^2v_{l_1}\circ A)\cdot A\cdot A^t\cdot (D^2v_{l_2}\circ A)\cdot A\\
		=\sum_{l_1,l_2=1}^{|F|}R_{jl_1}[A]R_{jl_2}[A]A^t\cdot (D^2v_{l_1}\circ A)\cdot (D^2v_{l_2}\circ A)\cdot A,
		\end{multline*}
	hence
		\begin{equation*}
		|D^2v_j|^2=\mathrm{tr}(D^2v_j\cdot D^2v_j)=\sum_{l_1,l_2=1}^{|F|}R_{jl_1}[A]R_{jl_2}[A](D^2v_{l_1}\circ A):(D^2v_{l_2}\circ A),
		\end{equation*}
		from which we get
		\begin{equation*}
		\sum_{j=1}^{|F|}|D^2v_j|^2=\sum_{j=1}^{|F|}|D^2v_j\circ A|^2.
		\end{equation*}
	\endproof
	\end{lem}
	
	\subsection{The Neumann problem}
	
	As we have done for the Steklov problem, we study the Neumann problem in $\phi(\Omega)$, i.e.,
\begin{equation}
\label{neumannphi}
\left\{
\begin{array}{ll}
\Delta^2u-\tau\Delta u=\lambda u, & {\rm in\ } \phi(\Omega),\\
\frac{\partial^2u}{\partial\nu^2}=0, & {\rm on\ }\partial\phi(\Omega),\\
\tau\frac{\partial u}{\partial\nu}-\mathrm{div}_{\partial\phi(\Omega)}(D^2u.\nu)-\frac{\partial\Delta u}{\partial\nu}=0,&{\rm on\ }\partial\phi(\Omega).
\end{array}\right.
\end{equation}
We consider the operator $\mathcal P^{\mathcal N}_{\phi}$ from $H^{2,\mathcal N}_{\phi}(\Omega)$ to $F(\Omega)$,
	defined by
	\begin{multline}
	\label{deltaphiN}
	\mathcal P^{\mathcal N}_{\phi}[u][\varphi]:=\int_{\Omega}(D^2(u\circ\phi^{-1})\circ\phi):(D^2(\varphi\circ\phi^{-1})\circ\phi)|\det D\phi|dx\\
		+\tau\int_{\Omega}(\nabla(u\circ\phi^{-1})\circ\phi)\cdot(\nabla(\varphi\circ\phi^{-1})\circ\phi)|\det D\phi|dx,\ 
		\forall u\in H^{2,\mathcal N}_{\phi}(\Omega), \varphi\in H^2(\Omega),
	\end{multline}
where
	$$
	H^{2,\mathcal N}_{\phi}(\Omega):=\left\{u\in H^2(\Omega):\int_{\Omega}u|\det D\phi|dx=0\right\},
	$$
	Moreover, for every $\phi\in\Phi(\Omega)$, we consider the map $\mathcal{J}^{\mathcal N}_{\phi}$ from $L^2(\Omega)$ to $H^2(\Omega)'$ defined by
	\begin{equation*}
	\mathcal{J}^{\mathcal N}_{\phi}[u][\varphi]:=\int_{\Omega}u\varphi|\det D\phi|d\sigma,\ \forall u\in L^2(\Omega),\varphi\in H^2(\Omega).
	\end{equation*}
We will think of the space $H^{2,\mathcal N}_{\phi}(\Omega)$ as endowed with the scalar product induced by (\ref{deltaphiN}). We denote by $\pi^{\mathcal N}_{\phi}$  the map from $H^2(\Omega)$ to $H^{2,\mathcal N}_{\phi}(\Omega)$ defined by
	$$
	\pi^{\mathcal N}_{\phi}(u)=:u-\frac{\int_{\Omega}u|\det D\phi|dx}{\int_{\Omega}|\det D\phi|dx},
	$$
and by $\pi_{\phi}^{\sharp,\mathcal N}$ the map from $H^2(\Omega)/\mathbb{R}$ onto $ H^{2,\mathcal N}_{\phi}(\Omega)$ defined by the equality $\pi^{\mathcal N}_{\phi}=\pi_{\phi}^{\sharp,\mathcal N}\circ p$. Clearly, $\pi_{\phi}^{\sharp,\mathcal N}$ is a homeomorphism, and we can recast problem (\ref{neumannphi}) as
$$
\lambda^{-1}u=W^{\mathcal N}_{\phi}u,
$$
where $W^{\mathcal N}_{\phi}:=(\pi_{\phi}^{\sharp,\mathcal N})^{-1}\circ(\mathcal P^{\mathcal N}_{\phi})^{-1}\circ\mathcal{J}^{\mathcal N}_{\phi}\circ i\circ\pi_{\phi}^{\sharp,\mathcal N}$ and $i$ is the canonical embedding of $H^2(\Omega)$
into $L^2(\Omega)$.
An analogue of Theorem \ref{thesame} can be stated also in this case. Therefore, we can compute Hadamard-type formulas for the
Neumann eigenvalues. This is contained in the following

\begin{thm}
		Let $\Omega$ be a bounded domain in $\mathbb{R}^N$ of class $C^1$. Let
		$F$ be a finite non-empty subset of $\mathbb{N}$. Let $\tilde{\phi}\in\Theta_{\Omega}[F]$ be such that $\partial\tilde{\phi}(\Omega)
		\in C^4$.
		Let $v_1,\dots,v_{|F|}$ be an orthonormal basis of the eigenspace associated with the eigenvalue $\lambda_F[\tilde{\phi}]$ of
		problem (\ref{neumannphi}) in $L^2(\tilde{\phi}(\Omega))$. Then
		\begin{equation*}
			d|_{\phi=\tilde{\phi}}(\Lambda_{F,s})[\psi]=-\lambda_F^{s-1}[\tilde{\phi}]\binom{|F|-1}{s-1}
\sum_{l=1}^{|F|}\int_{\partial\tilde{\phi}(\Omega)}\left(\lambda_Fv_l^2-\tau|\nabla v_l|^2-|D^2v_l|^2\right)\mu\cdot\nu d\sigma,
		\end{equation*}
		for all $\psi\in(C^2(\overline{\Omega}))^N$, where $\mu=\psi\circ\tilde{\phi}^{-1}$.
	
	\proof The proof is similar to that of Theorem \ref{duesette}.
	
	First of all we note that, by elliptic regularity theory, $v_1,\dots,v_{|F|}\in H^4(\tilde{\phi}(\Omega))$ (see \cite[\S 2.5]{ggs}).
	We set $u_l=v_l\circ\tilde{\phi}$ for $l=1,\dots,|F|$. For $|F|>1$ (case $|F|=1$ is similar), $s\leq|F|$, we have
	\begin{equation*}
	d|_{\phi=\tilde{\phi}}(\Lambda_{F,s})[\psi]
	=-\lambda_F^{s+2}[\tilde{\phi}]\binom{|F|-1}{s-1}\sum_{l=1}^{|F|}\mathcal P^{\mathcal N}_{\tilde{\phi}}\left[
		d|_{\phi=\tilde{\phi}}W^{\mathcal N}_{\phi}[\psi][p(u_l)]\right]\left[p(u_l)\right].
	\end{equation*}

	By standard calculus in normed spaces we have:
	\begin{multline*}
			\mathcal P^{\mathcal N}_{\tilde{\phi}}\left[\mathrm{d}|_{\phi=\tilde{\phi}}\left((\pi_{\phi}^{\sharp,\mathcal N})^{-1}\circ\left(
				\mathcal P^{\mathcal N}_{\phi}\right)^{-1}\circ\mathcal{J}^{\mathcal N}_{\phi}\circ i\circ\pi_{\phi}^{\sharp,\mathcal N}\right)[\psi][p(u_l)]\right]\left[p(u_l)\right]  \\ 
			= \mathcal P^{\mathcal N}_{\tilde{\phi}}\left[(\pi_{\tilde{\phi}}^{\sharp,\mathcal N})^{-1}\circ\left(\mathcal P^{\mathcal N}_{\tilde{\phi}}\right)^{-1}\circ
							\mathrm{d}|_{\phi=\tilde{\phi}}\left(\mathcal{J}^{\mathcal N}_{\phi}\circ i\circ\pi_{\phi}^{\sharp,\mathcal N}\right)[\psi][p(u_l)]\right]\left[p(u_l)\right]\\
			 						+\mathcal P^{\mathcal N}_{\tilde{\phi}}\left[\mathrm{d}|_{\phi=\tilde{\phi}}\left((\pi_{\phi}^{\sharp,\mathcal N})^{-1}\circ\left(
				\mathcal P^{\mathcal N}_{\phi}\right)^{-1}\right)[\psi]\circ\mathcal{J}^{\mathcal N}_{\tilde{\phi}}\circ i\circ\pi_{\tilde{\phi}}^{\sharp,\mathcal N}[p(u_l)]\right]\left[p(u_l)\right].
	\end{multline*}
	
	Now note that
	\begin{multline*}
		\mathcal P^{\mathcal N}_{\tilde{\phi}}\left[	(\pi_{\tilde{\phi}}^{\sharp,\mathcal N})^{-1}\circ\left(\mathcal P^{\mathcal N}_{\tilde{\phi}}\right)^{-1}\circ
							\mathrm{d}|_{\phi=\tilde{\phi}}\left(\mathcal{J}^{\mathcal N}_{\phi}\circ i\circ\pi_{\phi}^{\sharp,\mathcal N}\right)[\psi][p(u_l)]\right]\left[p(u_l)\right]\\
			=\int_{\tilde{\phi}(\Omega)}v_l^2\mathrm{div}\mu dy,
	\end{multline*}
	(see also Proposition \ref{vfiprop}) and
	\begin{multline*}
		\mathcal P^{\mathcal N}_{\tilde{\phi}}\left[\mathrm{d}|_{\phi=\tilde{\phi}}\left((\pi_{\phi}^{\sharp,\mathcal N})^{-1}\circ\left(
				\mathcal P^{\mathcal N}_{\phi}\right)^{-1}\right)[\psi]\circ\mathcal{J}^{\mathcal N}_{\tilde{\phi}}\circ i\circ\pi_{\tilde{\phi}}^{\sharp,\mathcal N}[p(u_l)]\right]\left[p(u_l)\right]\\
				=-\lambda_F^{-1}d|_{\phi=\tilde{\phi}}\left(\mathcal P^{\mathcal N}_{\phi}\circ\pi^{\mathcal N}_{\phi}\right)[\psi][u_l][\pi^{\mathcal N}_{\tilde{\phi}}(u_l)].
	\end{multline*}
	Using formula (\ref{formulafondamentale}) we obtain
	\begin{multline*}
				\mathcal P^{\mathcal N}_{\tilde{\phi}}\left[\mathrm{d}|_{\phi=\tilde{\phi}}\left((\pi_{\phi}^{\sharp,\mathcal N})^{-1}\circ\left(
				\mathcal P^{\mathcal N}_{\phi}\right)^{-1}\right)[\psi]\circ\mathcal{J}^{\mathcal N}_{\tilde{\phi}}\circ i\circ\pi_{\tilde{\phi}}^{\sharp,\mathcal N}[p(u_l)]\right]\left[p(u_l)\right]\\
=-\lambda_F^{-1}\int_{\partial\tilde{\phi}(\Omega)}\left(|D^2v_l|^2+\tau|\nabla v_l|^2\right)\mu\cdot\nu d\sigma
+\int_{\tilde{\phi}(\Omega)}\nabla(v_l^2)\cdot\mu dy.
\end{multline*}
To conclude, just observe that
\begin{equation*}
\int_{\tilde{\phi}(\Omega)}\nabla(v_l^2)\cdot\mu dy=
\int_{\partial\tilde{\phi}(\Omega)}v_l^2\mu\cdot\nu d\sigma
-\int_{\tilde{\phi}(\Omega)}(v_l^2)\mathrm{div}\mu dy.
\end{equation*}
	\endproof
\end{thm}

Now we can state the analogue of Theorem \ref{moltiplicatori} for problem (\ref{neumannphi}).
	\begin{thm}
	Let $\Omega$ be a bounded domain in $\mathbb{R}^N$ of class $C^1$. Let $F$ be a non-empty finite
	subset of $\mathbb{N}$. Let $\mathcal{V}_0\in]0,+\infty[$. Let $\tilde{\phi}\in V(\mathcal{V}_0)$ be such that
	$\partial\tilde{\phi}(\Omega)\in C^4$ and $\lambda_j[\tilde{\phi}]$ have a common value $\lambda_F[\tilde{\phi}]$ for all $j\in F$ and
	$\lambda_l[\tilde{\phi}]\neq\lambda_F[\tilde{\phi}]$
	for all $l\in\mathbb{N}\setminus F$. For $s=1,\dots, |F|$, the function $\tilde{\phi}$ is a critical point for $\Lambda_{F,s}$ on
	$V(\mathcal{V}_0)$ if and only if there exists an orthonormal basis $v_1,\dots,v_{|F|}$ of the eigenspace corresponding to the eigenvalue
	$\lambda_F[\tilde{\phi}]$ of problem (\ref{neumannphi}) in $L^2(\tilde{\phi}(\Omega))$, and a constant $c\in\mathbb{R}$ such that
	\begin{equation*}
	\sum_{l=1}^{|F|}\left(\lambda_Fv_l^2-\tau|\nabla v_l|^2-|D^2v_l|^2\right)=c,
	\mathrm{\ a.e.\ on\ }\partial\tilde{\phi}(\Omega).
	\end{equation*}
	\end{thm}

	We observe that Lemma \ref{propedeutico} holds for problem (\ref{neumannphi}) as well, since in the proof we have only used the rotation invariance of the Laplace operator. Then, we are led to the following
		
	\begin{thm}
	Let $\Omega$ be a bounded domain in $\mathbb{R}^N$ of class $C^1$. Let $\tilde{\phi}\in\Phi(\Omega)$ be such that
	$\tilde{\phi}(\Omega)$ is a ball. Let $\tilde{\lambda}$ be an eigenvalue of problem (\ref{neumannphi}) in $\tilde{\phi}(\Omega)$,
	and let $F$ be the set of $j\in\mathbb{N}$ such that $\lambda_j[\tilde{\phi}]=\tilde{\lambda}$.
	Then $\Lambda_{F,s}$ has a critical point at $\tilde{\phi}$ on $V(\mathcal{V}(\tilde{\phi}))$,
	for all $s=1,\dots,|F|$.
	\end{thm}

\section{The fundamental tone of the ball. The isoperimetric inequality}
\label{sec:5}
In the previous section we have shown that the ball is a critical point for all the elementary symmetric functions of the eigenvalues of problem (\ref{Steklov-Bi}). In this section we prove that the ball is actually a maximizer for the fundamental tone, that is 
\begin{equation}\label{ISO}
\lambda_2(\Omega)\leq\lambda_2(\Omega^{\ast}),
\end{equation}
where $\Omega^{\ast}$ is a ball such that $|\Omega|=|\Omega^{\ast}|$.

\subsection{Eigenvalues and eigenfunctions on the ball}

We characterize the eigenvalues and the eigenfunctions of (\ref{Steklov-Bi}) when $\Omega=B$ is the unit ball in $\mathbb R^N$ centered at the origin. It is convenient to use spherical coordinates $(r,\theta)$, where $\theta=(\theta_1,...,\theta_{N-1})$. The corresponding trasformation of coordinates is

\begin{eqnarray*}
x_1&=&r \cos(\theta_1),\\
x_2&=&r\sin(\theta_1)\cos(\theta_2),\\
\vdots\\
x_{N-1}&=&r\sin(\theta_1)\sin(\theta_2)\cdots\sin(\theta_{N-2})\cos(\theta_{N-1}),\\
x_N&=&r\sin(\theta_1)\sin(\theta_2)\cdots\sin(\theta_{N-2})\sin(\theta_{N-1}),
\end{eqnarray*}
with $\theta_1,...,\theta_{N-2}\in [0,\pi]$, $\theta_{N-1}\in [0,2\pi[$ (here it is understood that $\theta_1\in [0,2\pi[$ if $N=2$).

The boundary conditions of (\ref{Steklov-Bi}) in this case are written as

\begin{equation*}
\left\{\begin{array}{ll}
\frac{\partial^2 u}{\partial r^2 }_{|_{r=1}}=0 ,\\
\tau\frac{\partial u}{\partial r }-\frac{1}{r^2}{\Delta_S}\Big(\frac{\partial u}{\partial r}-\frac{u}{r}\Big)-\frac{\partial\Delta u}{\partial r}_{|_{r=1}}=\lambda u_{|_{r=1}},
\end{array}\right.
\end{equation*}
where $\Delta_S$ is the angular part of the Laplacian. It is well known that the eigenfunctions can be written as a product of a radial part and an angular part (see \cite{chas1} for details). The radial part is given in terms of ultraspherical modified Bessel functions and powertype functions. The ultraspherical modified Bessel functions $i_l(z)$ and $k_l(z)$ are defined as follows
\begin{eqnarray*}
i_l(z):=z^{1-\frac{N}{2}} I_{\frac{N}{2}-1+l}(z),\\
k_l(z):=z^{1-\frac{N}{2}} K_{\frac{N}{2}-1+l}(z),
\end{eqnarray*}
for $l\in\mathbb N$, where $I_{\nu}(z)$ and $K_{\nu}(z)$ are the modified Bessel functions of first and second kind respectively. We recall that $i_l(z)$ and all its derivatives are positive on $]0,+\infty[$ (see \cite[\S 9.6]{abram}). We recall that the Bessel functions $J_{\nu}$ and $N_{\nu}$ solve the Bessel equation
$$
z^2 y''(z)+zy'(z)+(z^2-\nu^2)y(z)=0,
$$ 
while the modified Bessel functions $I_{\nu}$ and $K_{\nu}$ solve the modified Bessel equation
$$
z^2 y''(z)+zy'(z)+(z^2+\nu^2)y(z)=0.
$$

We have the following 

\begin{thm}\label{eigenfunctions}
Let $\Omega$ be the unit ball in $\mathbb R^N$ centered at the origin. Any eigenfunction $u_l$ of problem (\ref{Steklov-Bi}) is of the form $u_l(r,\theta)=R_l(r)Y_l(\theta)$ where $Y_l(\theta)$ is a spherical harmonic of some order $l\in\mathbb N$ and
$$
R_l(r)=A_l r^l+B_l i_l({\sqrt{\tau}} r),
$$
where $A_l$ and $B_l$ are suitable constants such that
\begin{equation*}
B_l=\frac{l(1-l)}{\tau i''_l(\sqrt{\tau})}A_l.
\end{equation*}
Moreover, the eigenvalue $\lambda_{(l)}$ associated with the eigenfunction $u_l$ is delivered by formula
\begin{multline}\label{eigenvalues}
\lambda_{(l)}=l\Big((1 - l) l i_l(\sqrt{\tau}) + 
 \tau i_l''(\sqrt{\tau})\Big)^{-1}\Big[ 3 (l-1) l (l+N-2) i_l(\sqrt{\tau})\\
 - (l-1) \sqrt{
    \tau} \big(N-1 + 2 N l + 2l(l-2) l + \tau\big) i_l'(\sqrt{
     \tau})\\
		+ \tau \big((l-1) (l+2N-3)+ \tau\big) i_l''(
     \sqrt{\tau})\\
		+ (l-1) \tau\sqrt{\tau} 
i_l'''(\sqrt{\tau})\Big],
\end{multline}
for any $l\in\mathbb N$.

\proof
Solutions of problem (\ref{Steklov-Bi}) in the unit ball are smooth (see e.g., \cite[Theorem 2.20]{ggs}). We consider two cases: $\Delta u=0$ and $\Delta u\ne 0$.

Let $u$ be such that $\Delta u=0$. The Laplacian can be written in spherical coordinates as
$$
\Delta=\partial_{rr}+\frac{N-1}{r}\partial_r+\frac{1}{r^2}\Delta_S.
$$
Separating variables so that $u=R(r)Y(\theta)$ we obtain the equations
\begin{equation}\label{R}
R''+\frac{N-1}{r}R'-\frac{l(l+N-2)}{r^2}R=0
\end{equation}
and
\begin{equation}\label{Y}
\Delta_S Y=-l(l+N-2)Y.
\end{equation}
The solutions of equation (\ref{R}) are given by $R(r)=a r^l+b r^{2-N-l}$ if $l>0,N\geq 2$, and by $R(r)=a+b \log(r)$ if $l=0,N=2$. Since the solutions cannot blow up at $r=0$, we must impose $b=0$. The solutions of the second equation are the spherical harmonics of order $l$. Then $u$ can be written as 
$$u(r,\theta)=a_{l}r^{l}Y_{l}(\theta)$$
 for some $l\in\mathbb N$.

Let us consider now the case $\Delta u\ne 0$. We set $v=\Delta u$ and solve the equation 
$$
\Delta v=\tau v.
$$
By writing $v=R(r)Y(\theta)$ we obtain that $R$ solves the equation

\begin{equation}\label{equa2}
R''+\frac{N-1}{r}R'-\frac{l(l+N-2)}{r^2}R=\tau R,
\end{equation}
while $Y$ solves equation (\ref{Y}). Equation (\ref{equa2}) is the modified ultraspherical Bessel equation that is solved by the modified ultraspherical Bessel functions of first and second kind $i_l(\sqrt{\tau}r)$ and $k_l(\sqrt{\tau}r)$. Since the solutions cannot blow up at $r=0$, we must choose only $i_l(z)$ since $k_l(z)$ has a singularity at $z=0$. Then
$$
v(r,\theta)=b_{l_1}i_{l_1}(\sqrt{\tau}r)Y_{l_1}(\theta)
$$ 
for some $l_1\in\mathbb N$. Now $v=\frac{\Delta v}{\tau}=\Delta u$, that is $\Delta(v/\tau-u)=0$. This means that
\begin{equation}\label{solution-1}
u(r,\theta)=\frac{b_{l_1}}{\tau}i_{l_1}(\sqrt{\tau}r)Y_{l_1}(\theta)-c_{l_2}r^{l_2}Y_{l_2}(\theta)
\end{equation}
for some $l_2\in\mathbb N$.

Now we prove that the indexes $l_1$ and $l_2$ in (\ref{solution-1}) must coincide. This can be shown by imposing the boundary condition $\frac{\partial^2 u}{\partial r^2}_{|_{r=1}}=0$, which can be written as
\begin{equation}\label{agree}
b_{l_1}i''_{l_1}(\sqrt{\tau})Y_{l_1}(\theta)-c_{l_2}l_2 (l_2-1)Y_{l_2}(\theta)=0.
\end{equation}
If the two indexes do not agree, the coefficients of  $Y_{l_i}, i=1,2$ must vanish since spherical harmonics with different indexes are linearly independent on $\partial \Omega$. Since $i_{l_1}''(\sqrt{\tau})>0$, this implies $b_{l_1}=0$ and therefore $l_2=0$ or $l_2=1$. Then we have
\begin{equation}\label{eigenl}
u_l(r,\theta)=\Big(A_l r^l+B_l i_l(\sqrt{\tau}r)\Big)Y_l(\theta),
\end{equation}
with suitable constants $A_l,B_l$. In the case $l\ne 0,1$, again from the boundary condition (\ref{agree}) we have
\begin{equation}\label{rapporto2}
l(l-1)A_l+\tau i_l''(\sqrt{\tau})B_l=0,
\end{equation}
then $B_l=\frac{l(1-l)}{\tau i_l''(\sqrt{\tau})}A_l$. Note that the formula holds also in the case $l=0,1$ since  these indexes correspond to $B_l=0$.

Finally, let us consider the boundary condition
\begin{equation}\label{condition2}
\tau\frac{\partial u}{\partial r }-\frac{1}{r^2}{\Delta_S}\Big(\frac{\partial u}{\partial r}-\frac{u}{r}\Big)-\frac{\partial\Delta u}{\partial r}_{|_{r=1}}=\lambda u_{|_{r=1}}.
\end{equation}
Using in (\ref{condition2}) the representation of $u_l$ provided by formula (\ref{eigenl}), we get
\begin{multline*}
\Bigg[\Big(-\lambda + l \big((l-1) (l+N-2) + \tau\big)\Big) A_l+ \Big(-\big(3 l( l+N-2) + \lambda\big) i_l(\sqrt{\tau}) \\
- \sqrt{\tau} \big((N-1 - 2 N l - 2 (l-2) l - \tau) i_l'(\sqrt{\tau})+ (N-1) \sqrt{\tau} i_l''(\sqrt{\tau})\\
	+ \tau i_l'''(\sqrt{\tau})\big)\Big) B_l\Bigg]Y_l(\theta)=\lambda\left(A_l + B_l i_l(\sqrt{\tau})\right)Y_l(\theta).
\end{multline*}
Using equality (\ref{rapporto2}) we get that $u_l$ given by (\ref{eigenl}) is an eigenfunction of (\ref{Steklov-Bi}) on the unit ball. Moreover, as a consequence, we also get formula (\ref{eigenvalues}) for the associated eigenvalue. This concludes the proof.
\endproof
\end{thm}
 
We are ready to state and prove the following theorem concerning the first positive eigenvalue.

\begin{thm}
Let $\Omega$ be the unit ball in $\mathbb R^N$ centered at the origin. The first positive eigenvalue of (\ref{Steklov-Bi}) is $\lambda_2=\lambda_{(1)}=\tau$. The corresponding eigenspace is generated by $\left\{x_1,x_2,...x_N\right\}$.
\proof
By Theorem \ref{eigenfunctions}, $0=\lambda_{(0)}<\tau=\lambda_{(1)}$. We consider formula (\ref{eigenvalues}) with $l=2$. We have
\begin{multline}\label{lambda2}
\lambda_{(2)}=2\Big(\tau i_2''(\sqrt{\tau})-2 i_2(\sqrt{\tau})\Big)^{-1}\Big[6 N i_2(\sqrt{\tau}) - 
 \sqrt{\tau} (5N-1 + \tau)i_2'(\sqrt{\tau})\\
+ \tau (2N-1 + \tau) i_2''(\sqrt{\tau})+ \tau\sqrt{\tau} i_2'''(\sqrt{\tau})\Big].
\end{multline}
In order to prove that $\lambda_{(2)}>\tau$, we use some well-known recurrence relations between ultraspherical Bessel functions (see \cite[p. 376]{abram}),
\begin{eqnarray*}
i_l'(\sqrt{\tau}) &=& \frac{l}{\sqrt{\tau}} i_l(\sqrt{\tau}) + i_{l+1}(\sqrt{\tau}),\\
i_l''(\sqrt{\tau}) &=& \frac{l(l-1)}{\tau} i_l(\sqrt{\tau}) +\frac{l+2}{\tau}i_{l+1}(\sqrt{\tau}) + i_{l+2}(\sqrt{\tau}),\\
i_l'''(\sqrt{\tau}) &=& \frac{l(l-1)(l-2)}{\tau\sqrt{\tau}}i_{l}(\sqrt{\tau}) + \frac{l(2l+1)}{\tau} i_{l+1}(\sqrt{\tau}) +\frac{2(l+2)}{\sqrt{\tau}}i_{l+2}(\sqrt{\tau}) + i_{l+3}(\sqrt{\tau}).
\end{eqnarray*}
\normalsize
Using these relations in (\ref{lambda2}), we obtain an equivalent formula for $\lambda_{(2)}$,
\begin{multline*}
\lambda_{(2)}=2\Big(5\sqrt{\tau}i_3(\sqrt{\tau})+\tau i_4(\sqrt{\tau})\Big)^{-1}\Big[(10N-2+2\tau)i_2(\sqrt{\tau})\nonumber\\
+\big(2-10N+(7+10N)\sqrt{\tau}-2\tau+5\tau\sqrt{\tau}\big)i_3(\sqrt{\tau})\nonumber\\
+\tau(8+2N+\tau)i_4(\sqrt{\tau})+\tau\sqrt{\tau}i_5(\sqrt{\tau})\Big].
\end{multline*}
By well-known properties of the functions $I_{\nu}$ (see \cite[\S 9]{abram}), it follows that $i_l\geq i_{l+1}$ for all $l\in\mathbb N$. This implies

\begin{multline*}
(10N-2+2\tau)i_2(\sqrt{\tau})+\big(2-10N+(7+10N)\sqrt{\tau}-2\tau+5\tau\sqrt{\tau}\big)i_3(\sqrt{\tau})\\
+\tau(8+2N+\tau)i_4(\sqrt{\tau})+\tau\sqrt{\tau}i_5(\sqrt{\tau})\geq\Big(5\tau\sqrt{\tau}i_3(\sqrt{\tau})+\tau^2 i_4(\sqrt{\tau})\Big),
\end{multline*}
then
$$
\lambda_{(2)}\geq 2\tau>\tau=\lambda_{(1)}.
$$

Now it remains to prove that $\lambda_{(l)}$ is an increasing function of $l$ for $l\geq 2$. We adapt the method used in \cite[Theorem 3]{chas1}. We claim that for any smooth radial function $R(r)$ the Rayleigh quotient
$$
\mathcal Q(R(r)Y_l(\theta))=\frac{\int_{B}|D^2(R(r)Y_l(\theta))|^2+\tau|\nabla(R(r)Y_l(\theta))|^2dx}{\int_{\partial B}R(r)^2Y_l(\theta)^2d\sigma}
$$
is an increasing function of $l$ for $l\geq 2$. We consider the spherical harmonics to be normalized with respect to the $L^2(\partial B)$ scalar product. In particular, we have that the denominator $D[R(r)Y_l(\theta)]$ of ${\mathcal Q}(R(r)Y_l(\theta))$ is $R^2(1)$. For the numerator $N[R(r)Y_l(\theta)]$ of the Rayleigh quotient we have

\begin{eqnarray*}
N[R(r)Y_l(\theta)]&=&\int_0^1\Bigg(\frac{2k}{r^4}\Big(rR'-\frac{3}{2}R\Big)^2+\frac{k(k-N-1/2)}{r^4}R^2+\tau\frac{kR^2}{r^2}\Bigg)r^{N-1}dr\\
&&+\int_0^1\Big((R''^2)+\frac{N-1}{r^2}(R')^2+\tau(R')^2\Big)r^{N-1}dr,
\end{eqnarray*}
where $k=l(l+N-2)$. The above expression is increasing in $k$ for $k\geq N+1/2$ and since $k$ is an increasing function of $l$, we easily get that each term involving $l$ is an increasing function of $l$ for $l\geq 2$. Thus the claim above is proved.

For each $l\in\mathbb N$,
\begin{equation}\label{ray22}
\lambda_{(l)}=\inf\mathcal Q(u)=\inf \frac{\int_{B}|D^2 u|^2+\tau |\nabla u|^2 dx}{\int_{\partial B}u^2 d\sigma},
\end{equation}
where the infimum is taken among all functions $u$ that are $L^2(\partial B)-$orthogonal to the first $m-1$ eigenfunctions $u_i$ and $m\in\mathbb N$ is such that $\lambda_{(l)}=\lambda_m$ is the $m-$th eigenvalue of problem (\ref{Steklov-Bi}). The eigenfunctions $u_l$ are of the form $u_l=R_l(r)Y_l(\theta)$, and $u_l$ realizes the infimum in (\ref{ray22}). Then 
$$
\lambda_{(l)}=\mathcal Q(R_l(r)Y_l(\theta))\leq\mathcal Q(R_{l+1}(r)Y_l(\theta))\leq\mathcal Q(R_{l+1}(r)Y_{l+1}(\theta))=\lambda_{(l+1)},
$$
where the first inequality follows from the fact that $R_{l+1}(r)Y_l(\theta)$ is also orthogonal with respect to the $L^2(\partial B)$ scalar product to the first $m-1$ eigenfunctions $R_i(r)Y_i(\theta)$ for $i=1,...m-1$, and then it is a suitable trial function in (\ref{ray22}). The second inequality follows from the fact that the quotient $\mathcal Q(R(r)Y_l(\theta))$ is an increasing function of $l$, for $l\geq 2$. This concludes the proof.
\endproof
\end{thm}


\subsection{The isoperimetric inequality}

In this section we prove the isoperimetric inequality (\ref{ISO}). Actually, we prove a stronger result, that is a quantitative version of \eqref{ISO}. We adapt to our case a result of \cite{bradepruf}, where the authors prove a quantitative version of the Brock-Weinstock inequality for the Steklov Laplacian. We also refer to \cite{hannad,mel} where these kind of questions have been considered for the first time (see also \cite{brapra,fumagpra}).

Throughout this section $\Omega$ is a bounded domain of class $C^1$.
We recall the following lemma from \cite{bradepruf}.

\begin{lem}\label{isolem}
Let $\Omega$ be an open set with Lipschitz boundary and $p>1$. Then
\begin{equation*}
\int_{\partial\Omega}|x|^p d\sigma\geq\int_{\partial\Omega^*}|x|^p d\sigma \left(1+c_{N,p}\left(\frac{|\Omega\triangle\Omega^*|}{|\Omega|}\right)^2\right),
\end{equation*}
where $\Omega^*$ is the ball centered at zero with the same measure as $\Omega$, $\Omega\triangle\Omega^*$ is the symmetric difference of $\Omega$ and $\Omega^*$, and $c_{N,p}$ is a constant depending only on $N$ and $p$ given by
$$
c_{N,p}:=\frac{(N+p-1)(p-1)}{4}\frac{\sqrt[N]{2}-1}{N}\left(\min_{t\in[1,\sqrt[N]{2}]}t^{p-1}\right).
$$

\end{lem}

We also recall the following characterization of the inverses of the eigenvalues of (\ref{Steklov-Bi}) from  \cite{hile} (see also \cite{bandle}).

\begin{lem}\label{lemmavar}
Let $\Omega$ be a bounded domain of class $C^1$ in $\mathbb R^N$. Then the eigenvalues of problem (\ref{Steklov-Bi}) on $\Omega$ satisfy,
\begin{equation}\label{variational2}
\sum_{l=k+1}^{k+N}\frac{1}{\lambda_l(\Omega)}=\max\Bigg\{\sum_{l=k+1}^{k+N}\int_{\partial\Omega}v_l^2d\sigma\Bigg\},
\end{equation}
where the maximum is taken over the families $\{v_l\}_{l=k+1}^{k+N}$ in $H^2(\Omega)$ satisfying $\int_{\Omega}D^2v_i:D^2v_j+\tau\nabla v_i\cdot\nabla v_j dx=\delta_{ij}$, and $\int_{\partial\Omega}v_i u_jd\sigma=0$ for all $i=k+1,...,k+N$ and $j=1,2,...,k$, where $u_1,u_2,...,u_k$ are the first $k$ eigenfunctions of problem (\ref{Steklov-Bi}).
\end{lem}

For every open set $\Omega\in\mathbb R^N$ with finite measure, we recall the definition of Fraenkel asymmetry

\begin{equation*}
\mathcal A(\Omega):=\inf\left\{\frac{\|\chi_{\Omega}-\chi_{B}\|_{L^1(\mathbb R^N)}}{|\Omega|}\,:\,B\ {\rm ball\ with\ }|B|=|\Omega|\right\}.
\end{equation*}
The quantity $\mathcal A(\Omega)$ is the distance in the $L^1(\mathbb R^N)$ norm of a set $\Omega$ from the set of all balls of the same measure as $\Omega$. This quantity turns out to be a suitable distance between sets for the purposes of stability estimates of eigenvalues. Note that $\mathcal A(\Omega)$ is scaling invariant and $0\leq\mathcal A(\Omega)<2$. 
 
We are ready to prove the following

\begin{thm}\label{isoperimetrictheorem}
For every domain $\Omega$ in $\mathbb R^N$ of class $C^1$ the following estimate holds
\begin{equation}\label{quantitative}
\lambda_2(\Omega)\leq\lambda_2(\Omega^*)\left(1-\delta_N \mathcal A(\Omega)^2\right),
\end{equation}
where $\delta_N$ is given by
$$
\delta_N:=\frac{N+1}{8N}\left(\sqrt[N]{2}-1\right),
$$
and $\Omega^*$ is a ball with the same measure as $\Omega$.
\proof
Let $\Omega$ be a bounded domain of class $C^1$ in $\mathbb R^N$ with the same measure as the unit ball $B$. We consider in (\ref{variational2}) $l=2,...,N+1$ and $v_l=(\tau|\Omega|)^{-1/2}x_l$ as trial functions. The trial functions must have zero integral mean over $\partial\Omega$. This can be obtained by a change of coordinates $x=y-\frac{1}{|\partial\Omega|}\int_{\partial\Omega}y d\sigma$. Moreover, the functions $v_l$ satisfy the normalization condition of Lemma \ref{lemmavar}. Then $v_l$ are suitable trial functions to test in formula (\ref{variational2}). We get

$$
\sum_{l=2}^{N+1}\frac{1}{\lambda_l(\Omega)}\geq\frac{1}{\tau|\Omega|}\int_{\partial\Omega}|x|^2d\sigma.
$$
We use Lemma \ref{isolem} with $p=2$. This yields
\begin{multline*}
\sum_{l=2}^{N+1}\frac{1}{\lambda_l(\Omega)}\geq\frac{1}{\tau|\Omega|}\int_{\partial B}|x|^2d\sigma \left(1+c_{N,2}\left(\frac{|\Omega\triangle B|}{|\Omega|}\right)^2\right)\\
=\frac{N|B|}{\tau|B|}\left(1+c_{N,2}\left(\frac{|\Omega\triangle B|}{|\Omega|}\right)^2\right)=\frac{N}{\tau}\left(1+c_{N,2}\left(\frac{|\Omega\triangle B|}{|\Omega|}\right)^2\right)\\
=\sum_{l=2}^{N+1}\frac{1}{\lambda_l(B)}\left(1+c_{N,2}\left(\frac{|\Omega\triangle B|}{|\Omega|}\right)^2\right).
\end{multline*}

Suppose now that $\lambda_2(\Omega)\geq\frac{\tau}{2}$, otherwise estimate \eqref{quantitative} is trivially true, since $0\leq\mathcal A(\Omega)<2$. Since $\lambda_2(\Omega)\leq\lambda_l(\Omega)$ for all $l\geq 3$, the previous inequality and the definition of $\mathcal A(\Omega)$ yield
$$
\lambda_2(\Omega)\left(1+c_{N,2}\mathcal A(\Omega)^2\right)\leq\lambda_2(B).
$$
This implies \eqref{quantitative} with $\delta_N=\frac1 8\min\{1,\frac{N+1} N(\sqrt[N]{2}-1)\}$. Note that
$\min\{1,\frac{N+1} N(\sqrt[N]{2}-1)\}=\frac{N+1} N(\sqrt[N]{2}-1)$. This concludes the proof in the case $\Omega$ has the same measure as the unit ball.

The proof for general finite values of $|\Omega|$ relies on the well-known scaling properties of the eigenvalues. Namely, for all $\alpha>0$, if we write an eigenvalue of problem (\ref{Steklov-Bi}) as $\lambda(\tau,\Omega)$, we have
$$
\lambda(\tau,\Omega)=\alpha^3\lambda(\alpha^{-2}\tau,\alpha\Omega).
$$
This is easy to prove by looking at the variational characterization of $\lambda(\tau,\Omega)$ and $\lambda(\alpha^{-2}\tau,\alpha\Omega)$ and performing a change of variable $x\mapsto x/{\alpha}$ in the Rayleigh quotient (\ref{minmax}). This last observation concludes the proof of the theorem.
\endproof
\end{thm}

The isoperimetric inequality \ref{ISO} is an immediate consequence of Theorem \ref{isoperimetrictheorem}.

\begin{corol}
Among all bounded domains of class $C^1$ with fixed measure, the ball maximizes the first non-negative eigenvalue of problem (\ref{Steklov-Bi}), that is $\lambda_2(\Omega)\leq\lambda_2(\Omega^*)$, where $\lambda_2(\Omega)$ has been defined in (\ref{minmax}) and $\Omega^*$ is a ball with the same measure as $\Omega$.
\end{corol}

\begin{rem}
In \cite{bradepruf} the authors prove that the quantitative version of the Brock-Weinstock inequality that they find is sharp. We think that it would be of interest to consider the problem of the sharpness of inequality \ref{quantitative} as well. Unfortunately, the results of \cite{bradepruf} do not apply immediately to our case. Also, we do not discuss the sharpness here since we think it is out of the purposes of the present paper. Such a discussion will be part of a future work.
\end{rem}	


\section{Concluding remarks}
\label{sec:6}

Throughout this paper we have only considered problems (\ref{Steklov-Bi}) and (\ref{neumannweak}) with  $\tau>0$. If we set $\tau=0$, problem (\ref{Steklov-Bi}) reads
\begin{equation}\label{Steklov-Bi-2}
\left\{\begin{array}{ll}
\Delta^2 u =0 ,\ \ & {\rm in}\  \Omega,\\
\frac{\partial^2 u}{\partial \nu^2 }=0 ,\ \ & {\rm on}\  \partial \Omega,\\
-{\rm div_{\partial\Omega}}\big(D^2u.\nu\big)-\frac{\partial\Delta u}{\partial\nu}=\lambda u ,\ \ & {\rm on}\  \partial \Omega,
\end{array}\right.
\end{equation}
while problem (\ref{neumannweak}) reads
\begin{equation}\label{Neumann-Bi-2}
\left\{\begin{array}{ll}
\Delta^2 u = \lambda u,\ \ & {\rm in}\  \Omega,\\
\frac{\partial^2 u}{\partial \nu^2 }=0 ,\ \ & {\rm on}\  \partial \Omega,\\
{\rm div_{\partial\Omega}}\big(D^2u.\nu\big)+\frac{\partial\Delta u}{\partial\nu}=0 ,\ \ & {\rm on}\  \partial \Omega.
\end{array}\right.
\end{equation}

Problems (\ref{Steklov-Bi-2}) and (\ref{Neumann-Bi-2}) model free vibrating plates which are not subject to lateral tension. These problems  have a sequence of non-negative eigenvalues of finite multiplicity and the corresponding eigenfunctions form a orthonormal basis of $H^2(\Omega)$. The coordinate functions $x_1,...,x_N$ and the constants are eigenfunctions of both problems (\ref{Steklov-Bi-2}) and (\ref{Neumann-Bi-2}) corresponding to the eigenvalue $\lambda=0$, which has multiplicity $N+1$. Therefore, the first non-zero eigenvalue is the $(N+2)$-th eigenvalue.  

As we did in Theorem \ref{convergenzacompattaultimo}, we can define the family of problems
\begin{equation}\label{Neumann-Bi-2-eps}
\left\{\begin{array}{ll}
\Delta^2 u  = \lambda\rho_{\varepsilon} u,\ \ & {\rm in}\  \Omega,\\
\frac{\partial^2 u}{\partial \nu^2 }=0 ,\ \ & {\rm on}\  \partial \Omega,\\
{\rm div_{\partial\Omega}}\big(D^2u.\nu\big)+\frac{\partial\Delta u}{\partial\nu}=0 ,\ \ & {\rm on}\  \partial \Omega,
\end{array}\right.
\end{equation}
where $\rho_{\varepsilon}$ is defined as in (\ref{rhoeps}). We have the following theorem, whose proof can be easily done adapting that of Theorem \ref{convergenzacompattaultimo}.

\begin{thm}
Let $\Omega$ be a bounded domain in $\mathbb{R}^N$ of class $C^2$. Let $\rho_{\varepsilon}$ be defined as in (\ref{rhoeps}). Let $\lambda_j(\rho_{\varepsilon})$ be the eigenvalues of problem (\ref{Neumann-Bi-2-eps}) on $\Omega$ for all $j\in\mathbb N$. Let $\lambda_j$, $j\in\mathbb N$ denote the eigenvalues of problem (\ref{Steklov-Bi-2}) corresponding to the constant surface density $\frac{M}{|\partial\Omega|}$. Then we have $\lim_{\varepsilon\rightarrow 0}\lambda_j(\rho_{\varepsilon})=\lambda_j$ for all $j\in\mathbb N$.
\end{thm}

It is clear that a discussion similar to that of Section \ref{sec:4} can be carried out for problems (\ref{Steklov-Bi-2}) and (\ref{Neumann-Bi-2}) as well, by means of a change of the projections $\pi_{\phi}^{\mathcal S},\pi_{\phi}^{\mathcal N}$ according to the kernel. In particular, all the formulas in Section \ref{sec:4} remain true, by setting $\tau=0$. Then we have the following

\begin{thm}
	Let $\Omega$ be a domain in $\mathbb{R}^N$. Let $\tilde{\phi}\in\Phi(\Omega)$ be such that
	$\tilde{\phi}(\Omega)$ is a ball. Let $\tilde{\lambda}$ be an eigenvalue of problem (\ref{Steklov-Bi-2}) (problem (\ref{Neumann-Bi-2}) respectively) in $\tilde{\phi}(\Omega)$,
	and let $F$ be the set of $j\in\mathbb{N}$ such that $\lambda_j[\tilde{\phi}]=\tilde{\lambda}$.
	Then $\Lambda_{F,s}$ has a critical point at $\tilde{\phi}$ on $V(\mathcal{V}(\tilde{\phi}))$,
	for all $s=1,\dots,|F|$.
\end{thm}

Moreover, for problem (\ref{Steklov-Bi-2}), it is possible to identify the fundamental modes and the fundamental tone on the ball. We have the following

\begin{thm}
Let $\Omega=B$ be the unit ball in $\mathbb R^N$. The eigenfunctions of problem (\ref{Steklov-Bi-2}) are of the form
\begin{equation*}
u_l(r,\theta)=\left(A_lr^l+B_lr^{2+l}\right)Y_l(\theta),
\end{equation*}
for $l\in\mathbb N$, where $A_l$ and $B_l$ are suitable constants such that
$$
B_l=-\frac{l(l-1)}{(l+2)(l+1)}A_l.
$$
The eigenvalues $\lambda_{(l)}$ of problem (\ref{Steklov-Bi-2}) corresponding to the eigenfunctions $u_l(r,\theta)$ are delivered by the formula
\begin{equation*}
\lambda_{(l)}=\frac{l(l-1)\big(N+2Nl+(l-1)(2+3l)\big)}{1+2l}.
\end{equation*}
The first positive eigenvalue is
$$
\lambda_{N+2}=\lambda_{(2)}=2\left(N+\frac{8}{5}\right),
$$
and the corresponding eigenfunctions are
$$
u_2(r,\theta)=\big(6r^2-r^4\big)Y_2(\theta).
$$
\proof
The proof is similar to that of Theorem \ref{eigenfunctions}, from which it differs only for the use of biharmonic functions on the ball as solutions of the differential equation $\Delta^2 u=0$. For a characterization of biharmonic functions on the ball we refer to \cite{almansi1,almansi2,nico}.
\endproof
\end{thm}

We have an explicit form for the fundamental tone and for the corresponding eigenfunctions in the case of the unit ball which suggests how to construct trial functions for the Rayleigh quotient of $\lambda_{N+2}$. Unfortunately, if we want to use a function of the form $R(r)Y_2(\theta)$ as a test function as we did in Theorem \ref{isoperimetrictheorem} we must impose that $R(r)Y_2(\theta)$ is othogonal to the constants and to the coordinate functions with respect to the $L^2(\partial\Omega)$ scalar product and we can no more obtain this just by translating the domain $\Omega$.

We remark that functions of the form $R(r)Y_2(\theta)$ where $R(r)=6r^2-r^4$ for $r\in [0,1]$ and $R(r)=8r-3$ for $r>1$ are suitable trial functions for the annuli. Explicit computations show that, for example, in dimension 2 or 3 (where the formulas are less involved), the ball is a maximizer among radial domains with a fixed measure.
 
The results contained in this section suggest that the ball should be a maximizer also for problems (\ref{Steklov-Bi-2}) and (\ref{Neumann-Bi-2}). For what concerns problem (\ref{Neumann-Bi-2}), a characterization of the fundamental tone is still unavaiable. A deeper analysis of problems (\ref{Steklov-Bi-2}) and (\ref{Neumann-Bi-2}) will be part of a future work.

\section*{Acknowledgements}	
The authors are deeply thankful to Prof.\ Pier Domenico Lamberti who suggested the problem, and also for many useful discussions.
The authors gratefully aknowledge the anonymous referee for the careful reading of the manuscript and for their useful comments.
The authors acknowledge financial support from the research project
`Singular perturbation problems for differential operators' Progetto di Ateneo
of the University of Padova. The authors are members of the Gruppo Nazionale
per l'Analisi Matematica, la Probabilit\`a e le loro Applicazioni (GNAMPA) of
the Istituto Nazionale di Alta Matematica (INdAM).
The second author wishes to
thank the Center for Research and Development in Mathematics and Applications (CIDMA) of the University of Aveiro for the hospitality offered during the developement of part of the work.

\noindent {\small
Davide Buoso and Luigi Provenzano\\
Dipartimento di Matematica\\
Universit\`{a} degli Studi di Padova\\
Via Trieste,  63\\
35126 Padova\\
Italy\\
e-mail:	dbuoso@math.unipd.it \\
e-mail: proz@math.unipd.it

}
\end{document}